\documentclass[10pt,a4paper]{amsart}
\usepackage{amssymb,amsmath,amsfonts,amscd}
\usepackage{lmodern}
\usepackage[utf8]{inputenc} 
\usepackage[T1]{fontenc}    
\usepackage{url}            
\usepackage{booktabs}       
\usepackage{fancyhdr}
\usepackage{indentfirst}
\usepackage{graphicx}
\usepackage{subfigure}
\usepackage{wrapfig}
\usepackage{newlfont}
\usepackage{color}
\usepackage{latexsym}
\usepackage{csquotes}
\usepackage{amsthm}
\usepackage{mathtools}
\usepackage{nicefrac}
\usepackage{epstopdf}
\usepackage{caption}
\usepackage{bbm}
\usepackage[shortlabels]{enumitem}
\usepackage{tikz-cd}
\usepackage{microtype}      
\usepackage{lipsum}
\graphicspath{ {./images/} }

\newcommand{\Lsu}[2]{\bigwedge\nolimits^{#1}{#2}}
\newcommand{\Lgiu}[2]{\bigwedge\nolimits_{#1}{#2}}

\usepackage[numbers]{natbib}

\newcommand{\rn}[1]{{\mathbb R}^{#1}}
\newcommand{\R}{\mathbb R}

\newcommand{\he}[1]{{\mathbb H}^{#1}}

\newcommand{\mc}{\mathcal }

\newcommand{\eps}{\varepsilon}

\newcommand{\N}{\mathbb N}

\newcommand{\e}{\mathrm{Euc}}

\def \N {\mathbb{N}}
\def \R {\mathbb{R}}

\def \H {\mathbb{H}}

\def \Distr {\mathcal{D}}

\def \d {\mathrm{d}}

\def \de {\partial}

\def \e {\varepsilon }

\def \ker {\mathrm{Ker}}

\def \dsy {\displaystyle}
\def \Span {\mathrm{span}}

\def \h {\mathfrak{h}}

\def \emptyset {\varnothing}

\newtheorem{teo}{Theorem}[section]
\newtheorem{prop}[teo]{Proposition}
\newtheorem{coro}[teo]{Corollary}
\newtheorem{lemma}[teo]{Lemma}
\newtheorem{defi}[teo]{Definition}

\newtheorem{oss}[teo]{Remark}

\title[$L^p$-Hodge Decomposition in Sub-Riemannian Contact Manifolds]{$L^p$-Hodge Decomposition with Sobolev classes in Sub-Riemannian Contact Manifolds}

\author[A. Baldi, A. Rosa]{
 Annalisa Baldi, 
 Alessandro Rosa 
}

\begin{document}

\begin{abstract}

Let  $1<p<\infty$. In this article we establish an $L^p$-Hodge decomposition theorem on sub-Riemannian compact contact manifolds without boundary, related to the Rumin complex of differential forms. Given an $L^p$- Rumin's form,
we adopt an  approach in the spirit of   
 Morrey's book  \cite{morrey} (further performed in \cite{ISS}) to obtain a decomposition with higher regular ``primitives''   i.e.  that belong to suitable Sobolev classes. Our proof relies on recent results obtained in \cite{BFP2} and \cite{BTT}.
\end{abstract}




\maketitle

\section{Introduction}
A well-known result on complete Riemannian manifolds $(M,g)$ is the {Hodge decomposition } with respect to the Laplace-Beltrami-Hodge operator $\Delta_M:=\d\delta+\delta\d$, where $\d$ is the de Rham exterior differential and $\delta$ is its  $L^2$ formal adjoint operator. Roughly speaking, we can decompose a differential form into its exact, coexact and harmonic components.
The first results about Hodge decompositions date back to the works of Helmholtz \cite{helmholtz}, Hodge \cite{hodge1}, \cite{hodge2}, de Rham \cite{deRham1}, \cite{deRham2} and Weyl \cite{weyl}. Hodge and de Rham understood and established the connection between cohomology theory and the harmonic forms, while Weyl introduced the method of orthogonal projection in this context, applying it to the boundary value problems in the Euclidean setting.
In $1858$, Helmholtz \cite{helmholtz} proved a result that certain vector fields can be decomposed into the sum of an irrotational, i.e.\,curl-free, vector field and a solenoidal, i.e.\,divergence-free vector field, introducing the first proof of a version of the Hodge decomposition theorem. This result was used to solve boundary value problems arising from the study of hydrodynamic systems. In $1934$, Leray \cite{leray} used the $L^2$-harmonic projection to study the Navier-Stokes equations. In general, the Hodge decomposition  has great relevance in the study of PDEs, boundary value problems, quasi-conformal maps, and in the potential theory on Riemannian manifolds (we refer e.g. to \cite{schwarz}, \cite{duff_spencer},  \cite{ISS} and \cite{XiDong2}).

Let us denote by $\Lambda^h T^*M$ the $h-$exterior power of the cotangent boundle. In $1949$, Kodaira \cite{kodaira} proved the $L^2$-decomposition of $h$-forms in complete Riemannian manifolds as
$$L^2(M,\Lambda^h T^*M)=\mathcal{H}^{h,2}\oplus \overline{\d C_0^\infty(M,\Lambda^{h-1} T^*M)}\oplus\overline{\delta C_0^\infty(M,\Lambda^{h+1} T^*M)},$$
where $\mathcal{H}^{h,2}:=\ker\,\Delta_M\cap L^2(M,\Lambda^h T^*M)$ is the space of $L^2$-harmonic forms. Moreover, in $1991$, Gromov \cite{gromov_L2} proved that, under suitable assumption on the Laplacian $\Delta_M$, a stronger $L^2$-Hodge decomposition. 

The proof of the $L^p$-Hodge decomposition is considerably more complex. 
The study of the $L^p$-decomposion on compact Riemannian manifolds can be found, e.g., in Morrey'book  \cite{morrey}.  In 1995, Scott \cite{scott} proved the following $L^p$-Hodge decomposition  on compact Riemannian manifolds without boundary,  
\begin{equation}\label{eq: 2 gennaio prima}
	L^p(M,\Lambda^h T^*M)=\mathcal{H}^{h,p}\oplus \d W^{1,p}(M,\Lambda^{h-1} T^*M)\oplus\delta W^{1,p}(M,\Lambda^{h+1} T^*M)\,,
\end{equation}
for every $1<p<\infty$, 
where $\mathcal{H}^{h,p}:=\ker\,\Delta_M\cap L^p(M,\Lambda^h T^*M)$ is the space of $L^p$-harmonic forms. The general case $L^p$, for Sobolev classes in manifold with boundary, has been extensively studied in \cite{ISS}.

\medskip

The aim of this paper is to obtain a decomposition, analogous to \eqref{eq: 2 gennaio prima}, in the setting of compact sub-Riemannian contact manifolds without boundary. We replace the exterior differential $\d$ by a more suitable differential, denoted by $\d_C^M$, which acts on differential forms adapted to the contact geometry, the so-called {Rumin complex}, introduced by Rumin in \cite{rumin_jdg}. Roughly speaking, a contact manifold is given by the couple $(M,H)$, where $M$ is a smooth odd-dimensional (connected) manifold of dimension $2n+1$ and $H$ is the so called contact structure on $M$, that is, $H$ is a smooth distribution of hyperplanes which is maximally non-integrable: given $\theta^M$ a smooth 1-form defined on $M$ such that $H=\mathrm{Ker}\,\theta^M$, then $\d\theta^M$ restricts to a non-degenerate 2-form on $H$.   To be maximally non-integrable means, in some sense,
 that the contact subbundle $H$
is as far as possible from being integrable (see  \cite{mcduff_salamon}, Section 3.4). 
Therefore, a measure on a contact manifold $(M,H)$ can be defined through the non-degenerate top degree form $\theta^M\wedge(\d\theta^M)^n$.
There exists, in addition,  a unique vector field $\xi^M$ transverse to $\ker\,\theta^M$ (the so-called Reeb vector field) such that $\theta^M(\xi^M)=1$ and $\mathcal L_{\xi^M}=0$.

In this paper, following Rumin (see \cite{rumin_jdg}, p.288), we assume that there is a metric $g^M$ which is globally adapted to the symplectic form $\d\theta^M$. More specifically,  there exists
  an endomorphism $J$ of $H$ such that $J^2=-Id$,  $
\d\theta^M(X,JY)=-\d\theta^M(JX,Y)
$
for all $X, Y\in H$, 
and $\d\theta^M(X,JX)>0$ for all $X\in H\setminus\{0\}$. Then, if $X,Y\in H$ we define  $g_H(X,Y):=\d\theta^M(X,JY)$. Finally, we extend $J$ to $TM$ by setting $J\xi^M=\xi^M$ and setting $g^M(X,Y)=\theta^M(X)\theta^M(Y)+\d\theta^M(X,JY)$ for all $X,Y\in TM$.

The couple $(M,H)$ equipped with the Riemannian metric $g_H$  is called a \emph{sub-Riemannian} contact manifold and in the sequel it will be denoted by $(M,H,g^M)$, where $g^M$ is obtained as above.  In any sub-Riemannian contact manifold $(M,H,g^M)$, we can define a sub-Riemannian
distance $d_M$ (see e.g., [43]) inducing on $M$ the same topology of $M$ as a manifold.

It is well known, by Darboux theorem, that local models  of all
 contact manifolds are given by the so-called  Heisenberg groups (see e.g. \cite{mcduff_salamon}, p. 112).
The Heisenberg groups  are the simplest example of
a stratified nilpotent Lie group, endowed with a contact structure. We recall that
the Heisenberg group $\he n$  is the Lie group with  stratified nilpotent Lie algebra $\mathfrak h$
of step 2
\begin{eqnarray*}
\mathfrak h = \mathrm{span}\,\{X_1,\dots,X_n,Y_1,\dots,Y_n\}\oplus\mathrm{span}\,\{T\}:= \mathfrak h_1\oplus 
\mathfrak h_2,
\end{eqnarray*}
where the only nontrivial commutation rules are $[X_j,Y_j]=T$, $j=1,\dots,n$. 
We denote by
$\theta^{\mathbb H }$ the 1-form on $\he n$ such that $\ker\,\theta^{\mathbb H} = \mathfrak h_1$ and $\theta^{\mathbb H}(T)=1$.
We recall also that
 $\he n$ can be identified with $\rn{2n+1}$ through the exponential map. The stratification
 of the algebra induces a family of dilations $\delta_\lambda$ in the group via the exponential map,
\begin{equation}\label{uno}
\delta_\lambda=\lambda\textrm{ on }\mathfrak{h}_1,\quad \delta_\lambda=\lambda^2 \textrm{ on }\mathfrak{h}_2,
\end{equation}
which are analogous to the Euclidean homotheties.

Heisenberg groups  can be viewed as sub-Riemannian contact manifolds. If we choose on the contact sub-bundle of $\he n$ a left-invariant metric, it turns out that
the associated sub-Riemannian metric is also left-invariant. It is customary to call this
distance in $\he n$ a Carnot-Carath\'eodory distance (note  
that all left-invariant sub-Riemannian metrics on Heisenberg groups are bi-Lipschitz equivalent).

Differential forms on $\mathfrak{h}$  split into 2 eigenspaces under $\delta_\lambda$ (see \eqref{uno}), therefore the de Rham complex lacks scale invariance under the  anisotropic dilations $\delta_\lambda$.
 A way to recover scale invariance is to replace the de Rham complex by an intric  complex of differential forms, introduced by M. Rumin in \cite{rumin_jdg}. 
Smooth differential forms of degree $h$ are substituted by smooth sections of a vector bundle $E_0^h$, and Rumin's substitute for de Rham's exterior
differential $\d$ is a linear differential operator $\d_C^\mathbb H$ from sections of $E_0^h$
to sections of $E_0^{h+1}$ ($0\le h\le 2n+1$) such
that $(\d_C^\mathbb H)^2=0$.

%

 Rumin's construction of the intrinsic complex makes sense for arbitrary contact manifolds $(M,H)$ (see \cite{rumin_jdg}) and it  is invariant under contactomorphisms.
If $M$ has dimension $2n+1$, let $h=0,\ldots,2n+1$.  If $h\leq n$, the vector bundle $E_0^h$ is a subbundle of $\Lambda^h H^*$. If $h\geq n$, $E_0^h$ is a subbundle of $\Lambda^h H^*\otimes (TM/H)$ and
 we have:
\begin{itemize}
\item[i)] $(\d_C^M)^2=0$;
\item[ii)] the complex $(E_0^\bullet,\d_C^M)$  is homotopically equivalent to de Rham's complex
$ (\Omega^\bullet,\d)$; 
\item[iii)]  $\d_C^M: E_0^h\to E_0^{h+1}$ is a homogeneous differential operator in the 
horizontal derivatives
of order $1$ if $h\neq n$, whereas $\d_C^M: E_0^n\to E_0^{n+1}$ is a homogeneous differential operator in the 
horizontal derivatives
of order $2$.
\end{itemize}

 Given $(M,H,g^M)$, the scalar
product on $H$ determines  a norm on the
line bundle $T M/H$. Therefore for any $h$, the vector spaces $E_0^h$
are endowed with a scalar product. Using $\theta^M\wedge \d\theta^M$ as
a volume form, one gets $L^p$-norms on spaces of smooth Rumin's differential forms on $M$. 

We denote by $\ast$ the Hodge duality associated with the inner product in
$E_0^\bullet$ and the
volume form, and by $\delta_C^M$ the formal adjoint in $L^2$  of the operator $\d_C^M$, on $E_0^h$ we have $$\delta^M_C=(-1)^h\ast \d_C^M \ast $$  (see \cite{rumin_jdg} p. 288).
Starting from $\d^M_C$ and $\delta_C^M$,  Rumin has defined in \cite{rumin_jdg}  a suitable Hodge-type Laplacians on $M$ that  we can write as follows:
\begin{equation}\label{defi.RuminLapcontman intro}
\Delta_{M,h}:=\left\{
\begin{array}{ll}
\d^M_C\delta^M_C+\delta^M_C\d^M_C,\quad&\text{if $h\neq n,n+1$}, \\
(\d^M_C\delta^M_C)^2+\delta^M_C\d^M_C,\quad&\text{if $h=n$} \\
\d^M_C\delta^M_C+(\delta^M_C\d^M_C)^2,\quad&\text{if $h=n+1$}\,.
\end{array}
\right. 	
\end{equation}
We stress that, from the definition above, the Laplacians have order $2$ or $4$ depending on the order of the differential forms they act on. In the same paper, Rumin obtain the hypoellipticity of the contact complex, showing that $\Delta_M$ is maximal hypoelliptic. Moreover,  the following decomposition (see Corollary p. 290 in \cite{rumin_jdg}) for smooth compactly differential $h$-forms  is obtained,
\begin{equation}\label{25gennaio}
C_0^\infty(M, E_0^h)=\mathrm{Ker}\Delta_{M,h}\oplus\mathrm{Im} \Delta_{M,h}\,,
\end{equation}
(see also Corollary \ref{teo.HodgedecthCinf} below).

In this paper we show  an $L^p$-Hodge type decomposition akin to \eqref{eq: 2 gennaio prima}. This achievement relies on some recent results obtained in \cite{BFP2} and \cite{BTT} (also related to \cite{BFP3}). In particular, a crucial step in generalizing the Hodge-decomposition to Sobolev classes   relies on the validity of a Sobolev-Gaffney inequality, established in \cite{BTT} (on the setting of contact manifolds with bounded geometry). The Gaffney inequality allows to adapt to our context a variational approach, firstly  to obtain an 
$L^2$-decomposition (see \cite{morrey}, and Appendix A) and then to ensure the $L^p$ differentiability and integrability of $L^2$-decomposition.
Keeping in mind that $\he n$ is a local model of $M$, to introduce the notion of Sobolev space on $M$, we will make use of the analogous notion  in Heisenberg groups $\he n$. The definition of the so-called Folland-Stein Sobolev spaces  is associated with the stratification of the Lie algebra of $\he{n}$ and  we refer, e.g., to Section 4 of \cite{folland}. Roughly speaking, 
if $\ell\in\mathbb N$ and $p\ge 1$, the Sobolev space $W^{\ell,p}(\he n)$ can be defined as follows. Fix an orthonormal basis $\{W_i,\, i=1,\dots,2n+1\} $ of $\mathfrak h$ such
that $W_i\in \mathfrak h_1$ for $i=1,\dots,2n$ and $W_{2n+1}\in\mathfrak h_2$. 
We call {\sl homogeneous order} of a monomial in $\{W_i\} $ its degree of homogeneity with respect
to the group dilations $\delta_\lambda$, $\lambda>0$.
We say that a differential form on $\he n$ belongs to $W^{\ell,p}(\he n)$ if all its derivatives of homogeneous order $\le \ell$ along $\{W_i\} $
belong to $L^p(\he n)$. 
Using suitable-bounded charts, this notion extends to compact  sub-Riemannian contact manifolds $M$ (we refer to Section \ref{subsec:Sobspconmangeom} below for precise definitions and properties).

Our main results reads as follows.
\begin{teo}\label{teo.HodgedecthLpintro}
Let $(M,H,g^M)$ be a sub-Riemannian compact contact manifold without boundary and let $1\le h\le 2n$. If $1<p<\infty$, we have the decomposition 
\begin{equation}\label{eq.decompLphodgeintro}
L^p(M,E_0^h)=\mathcal{H}^h\oplus \d_C\Big(W^{s,p}(M,E_0^{h-1})\Big)\oplus \delta_C\Big(W^{r,p}(M,E_0^{h+1})\Big),
\end{equation}
where $s=2$ if $h=n+1$ and $s=1$ otherwise, and $r=2$ if $h=n$ and $r=1$ otherwise.
\end{teo}
Notice in the statment above the game of the exponents in the Sobolev spaces,  due to the different order of $\Delta_M$ when acting on forms of degree $n$ or $n+1$.

\medskip

The proof of this results needs of several preliminary steps. Firstly, we prove \eqref{eq.decompLphodgeintro} for $p=2$. To this aim, we first follow the line of the variational approach due to Morrey \cite{morrey}: if we denote by $\mathcal H$ the space of {harmonic} forms and, given a form in  $L^2\cap \mathcal H^\bot$, we obtain  the existence  of  a potential belonging to a suitable Sobolev space (depending on the order of the differential form we are considering).  Secondly, thanks to the $L^2$-global maximal hypoellipticity of $\Delta_M$ (which we prove for any $p$ in Theorem \ref{teo.globalmaxhypocont}), we are able to  show that the potential belongs to a higher-order Sobolev space. Thus, we obtain \eqref{eq.decompLphodgeintro} when $p=2$, passing also through
 the Hodge decomposition for smooth forms.
In order to prove \eqref{eq.decompLphodgeintro} for general $p$, our approach is now divided into two parts. First, we examine the case $2<p<\infty$, where we can use the information obtained for $p=2$ (recall that $M$ is compact). The case $1<p<2$ will follow from a duality argument, initially applied to smooth functions and then extended to $L^p$ by density (see also \cite{scott}). The proof of Theorem \ref{teo.HodgedecthLpintro} relies on three main ingredients: a recent result from \cite{BFP2} on Sobolev embedding, the $L^p$-maximal hypoellipticity of $\Delta_M$ proved in Theorem \ref{teo.globalmaxhypocont}, and a Gaffney inequality from \cite{BTT}, Theorem 5.4.

\medskip

The paper is organized as follows. In Section \ref{sec:Hnprop}, we review some definitions and properties of Heisenberg groups, and briefly discuss the Rumin complex. In particular, in Section \ref{subsec:SobSpHn}, we recall the definition of Sobolev spaces and related properties.
In Section \ref{subsec:Sobspconmangeom}, we clarify the notion and properties of manifolds with bounded geometry and the definition of Sobolev spaces. Additionally, we present some consequences of the Gaffney-type inequality, also discussed in Theorem \ref{teo.GaffIneqcontman}, and some results related to Sobolev inequalities in this setting. A global maximal hypoellipticity on contact manifolds with bounded geometry is proved in Section \ref{subsec:Rumlapmangeom}. As explained earlier, our proof of the decomposition \eqref{eq.decompLphodgeintro} initially passes through an $L^2$-{Hodge decomposition}, which is obtained in Section \ref{HodgeL2}, together with the Hodge decomposition for smooth forms in Section \ref{sec:finalcap}. Appendix A contains a large part of the proofs from Section \ref{HodgeL2}. In fact, although we deal with the Rumin complex, many arguments  are similar to those used in the Riemannian setting for the de Rham complex. However, in order to keep the paper self-contained, we decided to include them to highlight the main differences when dealing with differential forms in degree $n, n+1$, where the Laplacian $\Delta_M$ has order $4$ (additionally, we aim to emphasize the relevance of our Gaffney inequality in these proofs).
The main results of the paper are presented in Section \ref{sec:HodgedecLPformss}. The key ingredients in the proof of Theorem \ref{teo.HodgedecthLpintro} (compare with Theorem \ref{teo.HodgedecthLp}) are some regularity results obtained in Theorems \ref{teo.Lpreg2mpminfty} and \ref{teo.Lpreggeneral}. Furthermore, on compact manifolds, we prove in Section \ref{sec:poincineq} a more refined version of the inequalities stated in Theorem \ref{teo.GaffIneqcontman}. Appendix B contains some technical results related to Theorem \ref{teo.GaffIneqcontman}, which are useful for Section \ref{sec:poincineq}.


\section{Some preliminary results on Heisenberg groups and contact manifold}\label{sec:Hnprop}
We denote by $\H^n$ the $(2n+1)$-dimensional Heisenberg group, which is the simplest example of non-commutative Carnot group. We denote by $\h$ the stratified nilpotent Lie algebra of left-invariant vector fields of $\H^n$. A standard basis of $\h$ is given by, $i=1,\ldots,n$,
\begin{align}&W^\H_i:=\de_{\dsy x_i}-\frac{1}{2}x_{i+n}\de_{\dsy x_{2n+1}},\quad W_{i+n}^\H:=\de_{\dsy x_{i+n}}+\frac{1}{2}x_{i}\de_{\dsy x_{2n+1}},\nonumber\\
 &W_{2n+1}^\H:=\de_{\dsy x_{2n+1}},\nonumber\end{align}
where the only nontrivial commutation rules are $[W^\H_i,W^\H_{i+n}]=W^\H_{2n+1}$, for every $i=1,\ldots,n$.
The \textit{horizontal layer} $\h_1$ is the subspace of $\h$ spanned by $W_1^\H,\ldots,W_{2n}^\H$, so that we refer to $W_1^\H,\ldots,W_{2n}^\H$ as the \textit{horizontal derivatives}. Denoting by $\h_2$ the linear span of $W_{2n+1}^\H$, the two-step stratification of $\h$ is expressed by 
$$\h=\h_1\oplus\h_2.$$

A point $p\in\H^n$ is denoted by $p=(x,y,t)$ such that $x,y\in\R^n$ and $t\in\R$.
The stratification of $\h$ induces a family of nonisotropic dilations $\{\delta_\lambda\}_{\lambda>0}$ in $\H^n$ defined as, for any point  $p=(x,y,t)\in\H^n$,
$$\delta_\lambda(p)=(\lambda x,\lambda y,\lambda^2 t).$$
If we consider two points $p=(x,y,t), q=(\tilde x, \tilde y, \tilde t)\in\H^n$, then the non-commutative group operation is denoted by
$p\cdot\,q=\left(x+\tilde x, y+\tilde y, t+\tilde t +\frac{1}{2}\,\sum_{j=1}^n(x_j\tilde y_j-y_j \tilde x_j)\right).$

For any $q\in\H^n$, the \textit{left translation} $\tau_q:\H^n\to\H^n$ is defined as
$$p\mapsto \tau_q p:=q\cdot\,p.$$
Since $\H^n$ is a connected, simply connected, nilpotent Lie group, thus it is a locally compact group and it turns out that the left Haar measure coincides with the Lebesgue measure.

We consider the so-called Cygan-Konrányi norm, which is a homogeneous norm given by
$$\rho(p):=\left(\left(|x|^2+|y|^2\right)^2+16\,t^2\right)^\frac{1}{4}\,,$$
for every $p=(x,y,t)\in\H^n$, where $|\cdot|$ denotes the Euclidean norm. The associated gauge distance
$$d(p,q):=\rho(p^{-1}\cdot\, q),$$
for every $p,q\in\H^n$, is left-invariant with respect to left-translations, i.e.\,
$$d(\tau_q(p),\tau_q(p'))=d(p,p'),\quad\forall\,p,p',q\in\H^n.$$

The metric $d$ induces a topology on $\H^n$ generated by the Korányi balls defined as
\begin{equation*}
B(p,r):=\{q\in\H^n\,\,|\,\,d(p,q)<r\},
\end{equation*}
for every $p,q\in\H^n$ and $r>0$.

Note that the gauge norm is positively $\delta_\lambda$-homogeneous, i.e.\,$d(\delta_\lambda(p),\delta_\lambda(p'))=\lambda d(p,p')$, for every $p,p'\in\H^n, \lambda>0$, so that the Lebesgue measure of the ball $B(x,r)$ is $r^{2n+2}$ times the measure of the unit ball $B(e,1)$. Thus, the homogeneous dimension of $\H^n$ with respect to $\delta_\lambda$ is 
$$Q=2n+2.$$ 
We observe that the Hausdorff dimension of $\R^{2n+1}$ with respect to $d$ is $Q$, whereas the topological dimension of $\H^n$ is $2n+1$.

\subsection{Sobolev spaces in $\H^n$}\label{subsec:SobSpHn}
We will use the following notations for higher-order derivatives: If $J=(j_1,\cdots,j_{2n+1})$ is a multi-index, we set
\begin{equation*}
W^{J,\H}=(W^\H_1)^{j_1}\cdots(W^\H_{2n+1})^{j_{2n+1}},
\end{equation*}
and
$$|J|:=j_1+\cdots+j_{2n}+j_{2n+1},\quad\quad d(J):=j_1+\cdots+j_{2n}+2\,j_{2n+1},$$
that are, respectively, the order of the differential operator $W^{J,\H}$ and its degree of homogeneity with respect to $\H^n$ group dilations. By the Poincaré-Birkhoff-Witt theorem (see \cite{ricci}), the differential operators $(W^{J,\H})_J$ form a basis of the algebra of left-invariant differential operators in Heisenberg groups.
In particular,
if $d(J)=1$,  we can  consider the horizontal vector fields $\{W^\H_1,\ldots,W^\H_{2n}\}$.

We recall the notion of Folland-Stein Sobolev space (see \cite{folland_stein}).

\begin{defi}\label{defi.SobolevHn}
Let $U\subseteq \H^n$ be an open set, $1\leq p\leq \infty$ and $k\in\N$. 
Then, we define $W^{k,p}(U)$ as the space of all $f\in L^p(U)$ such that
$$W^{J,\H}f\in L^p(U),\quad\text{for all multi-indices $J$ with $d(J)\leq k$},$$
endowed with the norm
$$\|f\|_{W^{k,p}(U)}:=\sum_{d(J)\leq k}\|W^{J,\H}f\|_{L^p(U)}.$$
\end{defi}

Folland-Stein Sobolev spaces have the following properties (see \cite{BTT}, \cite{folland}, \cite{FSSC_houston}).
\begin{teo}\label{teo.denseSobspHn}
Let $U\subseteq \H^n$ be an open set, $1\leq p\leq \infty$ and $k\in\N$. Then,
\begin{enumerate}
\item $W^{k,p}(U)$ is a Banach space;
\item if $p<\infty$, $W^{k,p}(U)\cap C^\infty(U)$ is dense in $W^{k,p}(U)$;
\item if $p<\infty$, then $C_0^\infty(\H^n)$ is dense in $W^{k,p}(\H^n)$;
\item if $1<p<\infty$, then $W^{k,p}(U)$ is reflexive.
\end{enumerate}
\end{teo}

We recall the following inequalities and compact embeddings on Sobolev spaces (see \cite{BTT} Remark 2.7, and also \cite{lu_acta_sinica}, \cite{schwarz}) which implies the Ehrling's inequality in $\H^n$.
\begin{teo}\label{teo.compactembHn}
Let $B(e,1)$ be the Korányi ball in $\H^n$. Then, for every $1<p<\infty$, we have the compact embeddings
$$W^{2,p}(B(e,1))\hookrightarrow W^{1,p}(B(e,1))\hookrightarrow L^p(B(e,1)).$$
Moreover, we have the \textit{Ehrling's inequality}: if $f\in W^{2,p}(B(e,1))$, for every $\e>0$, there exists $c_\e>0$ such that
$$\|f\|_{W^{1,p}(B(e,1))}\leq \e\|f\|_{W^{2,p}(B(e,1))}+c_\e\|f\|_{L^p(B(e,1))}.$$
\end{teo}

\subsection{Multilinear algebra in $\H^n$}\label{subsec:multilinalgHn}

We set $\Lgiu{1}{\h}:=\h=\Span\{W^\H_1,\ldots,W^\H_{2n+1}\}$ and we denote by $\Lsu{1}{\h}$ the dual space of $\Lgiu{1}{\h}$, with the canonical orthonormal basis $\{\d x_1,\ldots,\d x_n,\d y_1,\ldots,\d y_n, \theta^\H\}=:\{\omega^\H_1,\ldots,\omega^\H_{2n+1}\}$, where
\begin{equation}\label{eq.contactformHn}
\theta^\H=\d t -\frac{1}{2}\sum_{j=1}^n(x_j\,\d y_j-y_j\,\d x_j)
\end{equation}
is called the \textit{contact form in $\H^n$}.

We write $\Lgiu{0}{\h}=\Lsu{0}{\h}:=\R$ and we take the exterior algebras $\Lgiu{\bullet}{\h}$ and $\Lsu{\bullet}{\h}$ of $\h$ and $\Lsu{1}{\h}$, respectively as
$$\Lgiu{\bullet}{\h}=\bigoplus_{h=0}^{2n+1} \Lgiu{h}{\h}\quad\text{and}\quad\Lsu{\bullet}{\h}=\bigoplus_{h=0}^{2n+1}\Lsu{h}{\h},$$
where
\begin{align}
\Lgiu{h}{\h}&=\Span\{W^\H_{i_1}\wedge\cdots\wedge W^\H_{i_h}\,|\,1\leq i_1<\cdots<i_h\leq 2n+1\}, \nonumber \\
\Lsu{h}{\h}&=\Span\{\omega^\H_{i_1}\wedge\cdots\wedge \omega^\H_{i_h}\,|\,1\leq i_1<\cdots<i_h\leq 2n+1\}, \nonumber
\end{align}
are said \textit{$h$-vectors} and \textit{$h$-covectors}, respectively. We denote by $\Theta_h$ and $\Theta^h$ their bases, for every $h=0,\ldots,2n+1$. Since $\h$ is a finite dimensional space, $\h$ can be endowed with an inner product $\langle\cdot,\cdot\rangle$, which we extend canonically to $\Lgiu{h}{\h}$ and $\Lsu{h}{\h}$, such that $\Theta_h$ and $\Theta^h$ are orthonormal bases.

The Hodge isomorphisms can be written as
$$\ast:\Lgiu{h}{\h}\stackrel{\sim}{\longrightarrow}\Lgiu{2n+1-h}{\h}\quad\text{and}\quad\ast:\Lsu{h}{\h}\stackrel{\sim}{\longrightarrow}\Lsu{2n+1-h}{\h},$$
and
\begin{align}
v&\wedge\ast w=\langle v,w\rangle\,W^\H_1\wedge\cdots\wedge W^\H_{2n+1},\quad\forall\,v,w\in\Lgiu{h}{\h}, \nonumber \\
\varphi&\wedge\ast \psi=\langle\varphi,\psi\rangle\,\omega^\H_1\wedge\cdots\wedge \omega^\H_{2n+1},\quad\forall\,\varphi,\psi\in\Lsu{h}{\h}.\nonumber
\end{align}

The  notion of left-invariance of a form is recalled in the next definition.
\begin{defi}\label{defi.leftinvform}
A $h$-form $\alpha$ on $\H^n$ is said to be left-invariant if
\begin{equation*}
\tau_q^\#\alpha=\alpha,\quad\text{for every $q\in\H^n$.},
\end{equation*}
where $\tau_q^\#\alpha$ is the pullback of $\alpha$ through the left translation $\tau_q$.
\end{defi}


The \textit{symplectic $2$-form} $\d \theta^\H$
is given by
$$\d \theta^\H=-\sum_{i=1}^n\,\omega_i^\H\wedge\omega_{i+n}^\H.$$
Thus, $-\d\theta^\H$ induces a symplectic structure on $\h_1$, and  $\{W_1^\H,\ldots,W_{2n}^\H\}$ is a symplectic basis of $\ker\theta^\H$.

\subsection{Rumin complex and Laplacians in $\H^n$}\label{subsec:RumcomplHn}
Let $0\le h\le 2n+1$. We denote by $\Omega^h(\H^n)$ the space of all smooth $h$-forms on $\H^n$ as the smooth sections of $\Lsu{h}{\h}$ (since it can be identified with the tangent space to $\H^n$ at the origin $e$ but also as a fiber bundle over $\H^n$ by left translations).

In the setting of Carnot groups, the usual de Rham differential operator $\d$ is, in general, non-isotropic with respect to the group dilations.
For example, in the case of $\H^n$, we have anisotropic dilations which split the Heisenberg algebra $\h$ in two eigenspaces with respect to the eigenvalues $\lambda$ and $\lambda^2$. Therefore, the de Rham differential $\d$ lacks on homogeneity and the de Rham complex lacks on scale invariance under these dilations.
In order to overcome this problem,  Rumin in \cite{rumin_jdg} considered a subcomplex of the de Rham complex that gives rise to a new differential, now commonly denoted by $\d_C$, that respects the stratification of the Lie algebra of a Carnot group, so that it is more suitable for studying differential forms on these anisotropic situations.

For the purposes of this paper, it is sufficient to recall only a few aspects of the Rumin complex and the related Rumin differential $\d_C$ on $\H^n$. Thus, we refer to \cite{rumin_jdg} and  \cite{rumin_grenoble} for a more detailed presentation (see also e.g. \cite{BFP2}).

For $h=0,\ldots,2n+1$, Rumin considers suitable subbundle of $\Lsu{h}{\h}$, denoted by $E_0^h$ and  the space of the smooth sections of $E_0^h$, which is still denoted by $E_0^h$, are said Rumin's $h$-forms. We denote by $\Xi^h_0\subseteq\Theta^h$ a basis of $E_0^h$.

We refer to  \cite{rumin_grenoble} for the detailed construction of the exterior differential $\d^\H_C:E_0^h\to E_0^{h+1}$ that makes $(E_0^\bullet,\d_C^\H)$ a complex that is homotopically equivalent to the de Rham's one. We limit ourselves to recall that $\d_C^\H$ is a left-invariant, homogeneous operator with respect to group dilations. Moreover, it is a first-order differential operator in the horizontal derivatives in degrees $h\neq n$ and it is a second-order differential horizontal operator in degree $n$. An explicit construction of $\d_C^\H$ in $\H^1$ and $\H^2$ is given in \cite{BF7}.

\begin{oss}
If $U\subseteq\H^n$ is an open set, $0\leq h\leq 2n+1$, $1\leq p\leq \infty$ and $m\in\N$, we denote by $W^{m,p}(U,E_0^h), C_0^\infty(U,E_0^h), C^\infty(U,E_0^h)$ the space of all the sections of $E_0^h$ such that their components with respect to a given left-invariant frame belong to the corresponding scalar spaces.
\end{oss}

We can extend the Theorem \ref{teo.compactembHn} to differential forms, obtaining the following compact embeddings for any $h$,
\begin{equation}\label{teo.compactembHnforms}W^{2,p}(B(e,1),E_0^h)\hookrightarrow W^{1,p}(B(e,1),E_0^h)\hookrightarrow L^p(B(e,1),E_0^h).\end{equation}
Thus,  if $\alpha\in W^{2,p}(B(e,1),E_0^h)$, for every $\e>0$ there exists $c_\e$ such that
$$\|\alpha\|_{W^{1,p}(B(e,1),E_0^h)}\leq \e\|\alpha\|_{W^{2,p}(B(e,1),E_0^h)}+c_\e\|\alpha\|_{L^p(B(e,1),E_0^h)}.$$





We denote by $\delta_C^\H$ the formal adjoint operator of the Rumin $\d^\H_C$ in $L^2(\H^n,E_0^h)$,
we have that $\delta^\H_C=(-1)^{(2n+1)h+1}\ast \d^\H_C \ast=(-1)^{h}\ast \d^\H_C \ast$ on $E_0^{h}$ (see e.g. \cite{rumin_jdg} p. 288).
We remind the reader that $\delta_C^\H$
is a left-invariant homogeneous differential operator in the horizontal variables, of order $1$ if $h\neq n+1$ and of order $2$ if $h=n+1$.

The  Laplacians defined by Rumin in $\H^n$ (\cite{rumin_jdg}) are defined as follows.
%
%
\begin{defi}\label{defi.RuminLapHn}
We define the {Rumin's Laplacians} $\Delta_{\H,h}$ on $E_0^h$ by setting
\begin{equation}\label{eq.LapRumHn}
\Delta_{\H,h}:=\left\{
\begin{array}{ll}
\d^\H_C\delta^\H_C+\delta^\H_C\d^\H_C,\quad&\text{if $h\neq n,n+1$}, \\
(\d^\H_C\delta^\H_C)^2+\delta^\H_C\d^\H_C,\quad&\text{if $h=n$} \\
\d^\H_C\delta^\H_C+(\delta^\H_C\d^\H_C)^2,\quad&\text{if $h=n+1$}
\end{array}
\right. 	
\end{equation}
\end{defi}

\begin{oss}\label{oss.orderLapRumHn}
 $\Delta_{\H,h}$ is a left-invariant homogeneous differential operator of order $2$ if $h\neq n,n+1$ and of order $4$ if $h=n,n+1$ with respect to group dilations.  Notice also that for $h=0$, $\Delta_0=-\sum_{i=1}^{2n} \left(W_i^\H\right)^2$ is the usual sub-Laplacian on functions.
Moreover, $\Delta_{\H,h}$ is \textit{self-adjoint}, for every $h$, and commutes with the Hodge operator, i.e.\,
$$\ast\Delta_{\H,h}=\Delta_{\H,2n+1-h}\ast.$$
\end{oss}

Since we fixed a basis $\Xi_0^h$ of $E_0^h$, for every $h$, we identify the set $\Distr'(\H^n,E_0^h)$ with $\Distr'(\H^n,\R^{N_h})$, where $N_h:=\dim\, E_0^h$. Hence,
 we can think to $\Delta_{\H,h}$ as a matrix-valued map
$$\Delta_{\H,h}=(\Delta_{\H,h}^{i j})_{i,j=1,\ldots,N_h}:\Distr'(\H^n,\R^{N_h})\to\Distr'(\H^n,\R^{N_h}).$$

In \cite{rumin_jdg}, it is proved that the Laplacians are hypoelliptic and {maximal hypoelliptic} operators (see Corollary p.$290$ or Theorem $3.1$ in \cite{rumin_gafa}) in the sense of \cite{HN} (see also \cite{ponge} for related argument).
 In
general, 
if $\mc L$ is a differential operator  on
$\mc D'(\he{n},\rn {N_h})$, then $\mc L$ is said hypoelliptic if 
for any open set $\mc V\subset \he{n}$ 
where $\mc L\alpha$ is smooth, then $\alpha$ is smooth in $\mc V$.
In addition, if $\mc L$ is
homogeneous of degree $a\in\mathbb N$,
we say that $\mc L$ is maximal hypoelliptic if
 for any  $\delta>0$ there exists $C=C(\delta)>0$ such that for any
homogeneous
polynomial $P$ in $W_1,\dots,W_{2n}$ of degree $a$
we have
$$
\|P\alpha\|_{L^{ 2}(\he n, \rn{N_h})}\le C
\left(
\|\mc L\alpha\|_{L^{ 2}(\he n, \rn{N_h})}+\|\alpha\|_{L^{ 2}(\he n, \rn{N_h})}
\right)
$$
for any $\alpha\in \mc C_0^\infty(B (e,\delta),\rn {N_h})$. 

In Theorem 4.1 in \cite{BFT3},   it is proved for $\Delta_{\H, h}$ that, for every $1<p<\infty$ and for any  $\delta>0$ there exists $C=C(\delta,p)>0$ such that 
\begin{equation}\label{eq:lp-hypo}
\|\alpha\|_{W^{2\ell,p}(\H^n,\rn {N_h})}\leq C\left(\|\Delta_{\H,h}\alpha\|_{L^p(\H^n,\rn {N_h})}+\|\alpha\|_{L^p(\H^n,\rn {N_h})}\right),\end{equation}
for any $\alpha\in C_0^\infty(B(e,\delta),\rn {N_h})$,
where $\ell=2$ if $h=n,n+1$ and $\ell=1$ if $h\neq n,n+1$.

%
%

\subsection{Rumin complex in contact manifolds}\label{subsec:Rumcomplcontman}
As mentioned above, the Rumin complex has been introduced on arbitrary contact manifolds. Here, we briefly recall only a few properties of the Rumin complex, and for a detailed construction, we refer again to \cite{rumin_jdg} (see also \cite{BFP2}, Section 2.3).

Let us start with the following definition.
\begin{defi}[Contactomorphism]\label{defi.contactomorphismm}
Let $(M_1,H_1)$ and $(M_2,H_2)$ be contact manifolds with contact forms given by $\theta^{M_1},\theta^{M_2}$, respectively. If $U_1\subseteq M_1,U_2\subseteq M_2$ are open sets and $f:U_1\to U_2$ is a diffeomorphism, then we say that $f$ is a \textit{contactomorphism} if there exists a non-vanishing function $g \in C^\infty(U_1,\R)$ such that
$$f^\#\theta^{M_2}=g\,\theta^{M_1}\quad\text{in $U_1$}.$$
\end{defi}
By Darboux theorem, 
any contact manifold of dimension $(2n+1)$ is locally contactomorphic to the Heisenberg group $\H^n$ (see, e.g., Theorem 5.1.5 in \cite{abraham_marsden}).

Let us denote by $(E_0^\bullet,\d_C^M)$ the Rumin complex on contact manifolds. Thus, we have that
\begin{itemize}
\item $(E_0^\bullet,\d_C^M)$ is homotopically equivalent to the de Rham complex $(\Omega^\bullet,\d)$;
\item $\d_C^M:E_0^h\to E_0^{h+1}$ is a differential operator of order $1$ if $h\neq n$, whereas it is a differential operator of order $2$ if $h=n$;
\item if $U\subseteq\H^n$ is an open set and $\phi:U\to M$ is a contactomorphism, denoting $V:=\phi(U)$, we have
\begin{enumerate}
\item $\phi^\# E_0^\bullet (V)= E_0^\bullet (U)$
\item $\d_C^\H\phi^\#=\phi^\#\d_C^M$.
\end{enumerate}
\end{itemize}

\subsection{Sobolev spaces on contact sub-Riemannian manifolds with bounded geometry and Sobolev-type inequalities}\label{subsec:Sobspconmangeom}
In  \cite{BFP3}, \cite{BFP2} and \cite{BTT}, the definition of Sobolev spaces has been considered in the class of sub-Riemannian contact manifolds with bounded geometry. Manifolds with bounded geometry
generalize the concept of compact manifolds and covering of compact manifolds. Some of the results recalled in this section hold in this more general context. 

In the sequel, we denote by $d_\H(\cdot,\cdot)$ the Korányi distance in $\H^n$ and by $B_\H(\cdot,\cdot)$ the associated Korányi ball.
\begin{defi}\label{defi.subRiemcontmanwboundgeom}
Let $(M,H,g^M)$ be a complete sub-Riemannian contact manifold and $k\in\N$. We say that $M$ has \textit{bounded $C^k$-geometry} if there exist $C_M, r>0$ such that, for every $x\in M$, there exists a contactomorphism $\phi_x:B_\H(e,1)\to M$ that satisfies
\begin{enumerate}
\item $B_M(x,r)\subseteq \phi_x(B_\H(e,1))$;
\item $\phi_x$ is a bi-Lipschitz map with constant $C_M$, i.e.\,
$$\frac{1}{C_M}d_\H(p,q)\leq d_M(\phi_x(p),\phi_x(q))\leq C_M d_\H(p,q),\quad\quad\text{for every $p,q\in B_\H(e,1)$};$$
\item for every $x,y\in M$, the transition maps $\phi_y^{-1}\circ\phi_x$ and their first $k$ derivatives with respect to left-invariant horizontal vector fields are bounded by $C_M$.
\end{enumerate}
\end{defi}

\begin{lemma}[see Lemma 4.11 in \cite{BFP2}]\label{lemma.coveringl}
Let $(M,H,g^M)$ be a sub-Riemannian contact manifold with bounded $C^k$-geometry, $k\in\N$. With the same notations as in Definition \ref{defi.subRiemcontmanwboundgeom}, there exists $\rho=\rho_r>0$ and an at most countable covering $\{B_M(x_j,\rho)\}_j$ of $M$ such that
\begin{enumerate}
\item each ball $B_M(x_j,\rho)$ is contained in the image of one of the contact charts of Definition \ref{defi.subRiemcontmanwboundgeom};
\item $B_M(x_j,\rho/5)\cap B_M(x_i,\rho/5)=\emptyset$, if $i\neq j$;
\item the covering is uniformly locally finite and there exists $N=N_M\in\N$ such that for each ball $B_M(x,\rho)$,
$$\text{card}\,\{m\in\N\,\,|\,\,B_M(x_m,\rho)\cap B_M(x,\rho)\neq\emptyset\}\leq N.$$
In addition, if $B_M(x_m,\rho)\cap B_M(x,\rho)\neq\emptyset$, then $B_M(x_m,\rho)\subseteq B_M(x,r).$
\end{enumerate}
\end{lemma}


Following \cite{BFP2}, we now recall the definition of Sobolev spaces on $M$.
\begin{defi}\label{defi.Sobolevcontmanbg}
Let $(M,H,g^M)$ be a sub-Riemannian contact manifold with $C^k$-bounded geometry, $k\in\N$, and let $\{\chi_j\}_{j\in\N}$ be a partition of the unity subordinated to the atlas $\{B_M(x_j,\rho),\phi_{x_j}\}_{j\in\N}$ of Lemma \ref{lemma.coveringl}. If $\alpha$ is a section of $E_0^h$ on $M$, we say that $\alpha\in W^{m,p}(M,E_0^\bullet)$, for $m=0,\ldots,k-1$ and $p\geq 1$, if
$$\phi^\#_{x_j}(\chi_j\alpha)\in W^{m,p}(\H^n,E_0^\bullet),\quad\text{for every $j\in\N$}$$
and
\begin{equation}\label{eq.normMsobolev}
\|\alpha\|_{W^{m,p}(M,E_0^\bullet)}:=\left(\sum_{j\in\N}\|\phi^\#_{x_j}(\chi_j\alpha)\|^p_{W^{m,p}(\H^n,E_0^\bullet)}\right)^{\frac{1}{p}}< +\infty.
\end{equation}
We stress the fact that $\phi_{x_j}^{-1}(\mathrm{supp}\,\chi_j)\subseteq B_\H(e,1)$, thus $\phi^\#_{x_j}(\chi_j\alpha)$ is compactly supported in $B_\H(e,1)$ and it can be continued by zero on all $\H^n$. For sake of simplicity let us set
$$
\phi_j:=\phi_{x_j}\,.
$$
\end{defi}

From now on, we will consider contact manifolds with $C^k$-bounded geometry with $k\in\N$ sufficiently big enough so that every Sobolev spaces  that we consider, make sense as in Definition \ref{defi.Sobolevcontmanbg}.

One can prove (see \cite{BFP2}, Proposition $4.13$) that different choices of the covering and charts of $M$ give rise to equivalent norms of \eqref{eq.normMsobolev}, thus the definition of $W^{m,p}(M,E_0^\bullet)$ does not depend on the atlas $\{B_M(x_j,\rho),\phi_{x_j}\}_{j\in\N}$.

%

As proved in \cite{BTT}, Remark 3.4, the norm of $W^{0,p}(M,E_0^\bullet)$ is equivalent to the norm of $L^{p}(M,E_0^\bullet)$ associated to the volume form $\d V:=\theta^M\wedge(\d\theta^M)^n$. 



If $\alpha\in W^{m,p}(M,E_0^h)$ is supported in $\phi_j\big(B_\H(e,\eta/C_M)\big)$, then the norms
$$\|\alpha\|_{W^{m,p}(M,E_0^h)}\quad\text{and}\quad\|\phi_j^\#\alpha\|_{W^{m,p}(\H^n,E_0^h)}$$
are equivalent, with constants independent on $j$.
Moreover,  we have the following result.
\begin{teo}[see Theorem 5.4 in \cite{BTT}]\label{teo.GaffIneqcontman}
Let $(M,H,g^M)$ be a sub-Riemannian contact manifold with $C^k$-bounded geometry and without boundary, $k\ge 3$. Let $1\leq h\leq 2n$, $1< p<\infty$. Then, there exists a geometric constant $C>0$ such that, 
\begin{itemize}
\item[i)]  for every $\alpha\in W^{1,p}(M, E_0^h)$, with $h\neq n, n+1$, we have
\begin{equation}\label{eq.GaffIneqh}
\|\alpha\|_{W^{1,p}(M,E_0^h)}\leq C\Big(\|\d_C^M\alpha\|_{L^{p}(M,E_0^{h+1})}+\|\delta^M_C\alpha\|_{L^{p}(M,E_0^{h-1})}+\|\alpha\|_{L^p(M,E_0^h)}\Big);
\end{equation}
\item[ii)]  for every $\alpha\in W^{2,p}(M, E_0^n)$, we have
\begin{equation}\label{eq.GaffIneqn}
\|\alpha\|_{W^{2,p}(M,E_0^n)}\leq C\Big(\|\d_C^M\alpha\|_{L^{p}(M,E_0^{n+1})}+\|\d^M_C\delta^M_C\alpha\|_{L^{p}(M,E_0^{n})}+\|\alpha\|_{L^p(M,E_0^n)}\Big);
\end{equation}
\item[iii)]  for every $\alpha\in W^{2,p}(M, E_0^{n+1})$, we have
\begin{equation}\label{eq.GaffIneqnp1}
\|\alpha\|_{W^{2,p}(M,E_0^{n+1})}\leq C\Big(\|\delta^M_C\d_C^M\alpha\|_{L^{p}(M,E_0^{n+1})}+\|\delta^M_C\alpha\|_{L^{p}(M,E_0^{n})}+\|\alpha\|_{L^p(M,E_0^{n+1})}\Big).
\end{equation}
\end{itemize}
\end{teo}
As we will see in the sequel, this inequality plays a crucial role in establishing the Hodge decompositions for manifolds presented in Sections \ref{HodgeL2} and \ref{sec:HodgedecLPformss}.

%

\subsubsection{Sobolev embeddings on contact manifolds}\label{sec:Sobembconman}

Let $(M,H,g^M)$ be a sub-Riemannian contact manifold with bounded geometry. An approximate smoothing homotopy formula has been described and proved in \cite{BFP2} (see Theorem $1.5$ in \cite{BFP2}).

\begin{teo}[Theorem 1.5 in \cite{BFP2}]\label{teo.Homformcontman}
Let $(M,H,g^M)$ be a $(2n+1)$-dimensional sub-Riemannian contact manifold with $C^k$-bounded geometry with $k\geq 3$. Let $1<p\leq q<\infty$ that satisfy the condition 
\begin{equation}\label{eq.assumIhpqn}
\frac{1}{p}-\frac{1}{q}\leq\left\{
\begin{array}{l}
\frac{1}{2n+2},\quad\text{if $h\neq n+1$}, \\
\frac{2}{2n+2},\quad\text{if $h= n+1$}.
\end{array}\right.
\end{equation}
Then, for $1\leq h\leq 2n+1$, there exist operators $S_M$ and $T_M$ on $h$-forms on $M$ such that
 are bounded from $W^{m,p}_M$ to $W^{m,q}_M$, for every $0\leq m\leq k-1$ and such that
$$Id=\d_C^M T_M+T_M\d_C^M+S_M.$$
Furthermore,  $S_M\d_C^M=\d_C^M S_M$ and $S_M$ and $T_M$ are bounded from $W^{m-1,p}_M$ to $W^{m,p}_M$ if $m\geq 1$ and $h\neq n+1$, and they are bounded from $W^{m-2,p}_M$ to $W^{m,p}_M$ if $m\geq 2$ and $h=n+1$.
(See formulae $(64), (65)$ in \cite{BFP2} for an explicit definition of $S_M$ and $T_M$).
\end{teo}

A starightforward consequence is the following result.
\begin{prop}\label{teo.Sobembthcontman}
Let $(M,H,g^M)$ be a $(2n+1)$-dimensional sub-Riemannian contact manifold with $C^k$-bounded geometry with $k\geq 3$. Let $1<p\leq q<\infty$ that satisfy \eqref{eq.assumIhpqn}. Then, for $1\leq h\leq 2n+1$, we have the following continuous embedding
$$W^{\ell,p}(M,E_0^h)\hookrightarrow L^q(M,E_0^h),$$
where $\ell=1$ if $h \neq n,n+1$ and $\ell=2$ if $h=n$ or $h=n+1$.
\end{prop}
\begin{proof}
Let $\gamma\in W^{\ell,p}(M,E_0^h)$. Then, by homotopy formula in Theorem \ref{teo.Homformcontman}, 
$$\gamma=\d_C^MT_M\gamma+T_M\d_C^M\gamma+S_M\gamma.$$
Thus,
$$\|\gamma\|_{L^q}\leq\|\d_C^MT_M\gamma\|_{L^q}+\|T_M\d_C^M\gamma\|_{L^q}+\|S_M\gamma\|_{L^q}.$$
Recalling that $\d_C^M$ is a differential operator of order $2$ if $h=n,n+1$ and of order $1$ otherwise, and applying the properties of $T_M$, we have
\begin{itemize}
\item if $h\neq n, n+1$,
\begin{equation*}
\|\d_C^M T_M\gamma\|_{L^q}\leq c \|T_M\gamma\|_{W^{1,q}}\leq c \|\gamma\|_{W^{1,p}},\quad\quad\|T_M\d_C^M\gamma\|_{L^q}\leq c \|\d_C^M\gamma\|_{L^p}\leq c \|\gamma\|_{W^{1,p}}.
\end{equation*}
\item if $h=n+1$,
\begin{equation*}
\|\d_C^M T_M\gamma\|_{L^q}\leq c \|T_M\gamma\|_{W^{2,q}}\leq c \|\gamma\|_{W^{2,p}},\quad\quad\|T_M\d_C^M\gamma\|_{L^q_M}\leq c \|\d_C^M\gamma\|_{L^p}\leq c \|\gamma\|_{W^{1,p}}.
\end{equation*}
\item if $h=n$,
\begin{equation*}
\|\d_C^M T_M\gamma\|_{L^q}\leq c \|T_M\gamma\|_{W^{1,q}}\leq c \|\gamma\|_{W^{1,p}},\quad\quad\|T_M\d_C^M\gamma\|_{L^q}\leq c \|\d_C^M\gamma\|_{L^p}\leq c \|\gamma\|_{W^{2,p}}.
\end{equation*}
\end{itemize}
Using also the properties of $S_M$, we obtain that
$$\|\gamma\|_{L^q}\leq c\|\gamma\|_{W^{\ell,p}},$$
completing the proof.
\end{proof}

\begin{oss}\label{oss.SobembstimaCQ}
Set 
\begin{equation*}
C_Q:=\left\{\begin{array}{ll}
2n+2,\quad&\text{if $h\neq n+1$}, \\
n+1,\quad&\text{if $h= n+1$},
\end{array}\right.
\end{equation*}
and let $\ell=1$ if $h \neq n,n+1$ and $\ell=2$ if $h=n$ or $h=n+1$.
By the previous proposition, for any $\gamma\in W^{\ell,p}$ and $p\leq q\leq p+\frac{p}{C_Q}$ we have
$$\|\gamma\|_{L^q}\leq c\|\gamma\|_{W^{\ell,p}}\,.$$
\end{oss}

\medskip

In the setting of compact contact manifolds,  the following compact embeddings also hold true.
\begin{oss}\label{teo.compactembcontact}
Let $(M,H,g^M)$ be a compact sub-Riemannian contact manifold. Then, for every $1< p<\infty$, we have the following compact embeddings:
$$W^{2,p}(M,E_0^h)\hookrightarrow W^{1,p}(M,E_0^h)\hookrightarrow L^p(M,E_0^h),$$
for every $0\leq h\leq 2n+1$ and $m\in\N$.
\end{oss}
\begin{proof}
At first, we recall that the covering of $M$ is finite since $M$ is compact.
Let us prove, for example, that
$$W^{1,p}(M,E_0^h)\hookrightarrow L^p(M,E_0^h).$$
Let $(\alpha_k)_{k\in\N}$ be a sequence in $W^{1,p}(M,E_0^h)$ that is bounded in norm by a constant $C>0$, i.e.\,
$$\sup_{k\in\N}\,\|\alpha_k\|^p_{W^{1,p}(M,E_0^h)}\leq C.$$
Hence, using the same covering $\{(B_j,\phi_j)\}_{j}$ that was chosen for the validity of Theorem \ref{teo.GaffIneqcontman}, with $\eta/C_M<1$, we have for every $j$
\begin{align*}
\sup_{k\in\N}\,\left\|\phi^\#_{j}(\alpha_k|_{B_j})\right\|^p_{W^{1,p}(B_\H(e,\eta/C_M),E_0^h)}&\approx\sup_{k\in\N}\,\left\|\alpha_k|_{B_j}\right\|^p_{W^{1,p}(B_j,E_0^h)}\\
&\leq\sup_{k\in\N}\,\|\alpha_k\|^p_{W^{1,p}(M,E_0^h)}\leq C\,.
\end{align*}
Thus, $\big(\phi^\#_{j}(\alpha_k|_{B_j})\big)_{k\in\N}$ is a bounded sequence in $W^{1,p}(B_\H(e,1),E_0^h)$ and we can apply  \eqref{teo.compactembHnforms}. Hence, there exists a form $\gamma_j\in L^p(B_\H(e,1),E_0^h)$ and a subsequence $(m_k)_{k\in\N}$ such that we have
\begin{equation}\label{eq.conveunifindex}
\phi^\#_{j}(\alpha_{m_k}|_{B_j})\xrightarrow{k \rightarrow +\infty}\gamma_j\quad \text{in}\quad L^p(B_\H(e,1),E_0^h),
\end{equation}
or, equivalently,
\begin{equation}\label{eq.conveunifindex2}
\alpha_{m_k}|_{B_j}\xrightarrow{k \rightarrow +\infty}(\phi_{j}^{-1})^\#\gamma_j\quad \text{in}\quad L^p(B_j,E_0^h).
\end{equation}
We can extract at every step $j$, a subsequence of indexes from the previous $(j-1)$-th sequence so that the convergences \eqref{eq.conveunifindex} hold with the same sequence of indexes that we continue to label $(m_k)_{k\in\N}$ for every $j$.

Now, we want to define a form $\alpha$ on $M$ such that on the local chart $(B_j,\phi_{j})$ it is given by the pullback of $\gamma_j$ by $\phi_{j}^{-1}$, i.e.\,we want that $\alpha|_{B_j}=(\phi_{j}^{-1})^\#\gamma_j$, for every $j$. The form $\alpha$ exists since we have the following gluing property:
\begin{align*}
\|(&\phi_{j}^{-1})^\#\gamma_j|_{B_i}-(\phi_{i}^{-1})^\#\gamma_i|_{B_j}\|^p_{L^p(B_i\cap B_j,E_0^h)}\\
&\leq\|(\phi_{j}^{-1})^\#\gamma_j|_{B_i}-\alpha_{m_k}|_{B_i\cap B_j}\|^p_{L^p(B_i\cap B_j,E_0^h)}\\
&\quad\quad\quad+\|(\phi_{i}^{-1})^\#\gamma_i|_{B_j}-\alpha_{m_k}|_{B_i\cap B_j}\|^p_{L^p(B_i\cap B_j,E_0^h)}\\
&\leq \|(\phi_{j}^{-1})^\#\gamma_j-\alpha_{m_k}|_{B_j}\|^p_{L^p(B_j,E_0^h)}\\
&\quad\quad\quad+\|(\phi_{i}^{-1})^\#\gamma_i-\alpha_{m_k}|_{B_i}\|^p_{L^p(B_i,E_0^h)}\xrightarrow{k\to+\infty} 0,
\end{align*}
showing that $(\phi_{j}^{-1})^\#\gamma_j|_{B_i}\equiv(\phi_{i}^{-1})^\#\gamma_i|_{B_j}$, for every $i,j$.

We want to prove that $\alpha_{m_k}\xrightarrow{k \rightarrow +\infty}\alpha$ in $L^p(M,E_0^h)$.
At first, keeping also in mind that the sum on $j$ is finite, we observe that
\begin{align*}
\|\alpha\|^p_{L^p(M,E_0^h)}\leq\sum_{j}\|\alpha|_{B_j}\|^p_{L^p(B_j,E_0^h)}\leq c\sum_{j}\|\gamma_j\|^p_{L^p(B_\H(e,1),E_0^h)}<\infty.
\end{align*}
Thus, $\alpha\in L^p(M,E_0^h)$. Moreover
\begin{align*}
\|\alpha_{m_k}-\alpha\|^p_{L^p(M,E_0^h)}&\leq\sum_{j\,\text{finite}}\|\alpha_{m_k}|_{B_j}-\alpha|_{B_j}\|^p_{L^p(B_j,E_0^h)}\\
&=\sum_{j\,\text{finite}}\|\alpha_{m_k}|_{B_j}-(\phi_{j}^{-1})^\#\gamma_j\|^p_{L^p(B_j,E_0^h)}\xrightarrow{k \rightarrow +\infty} 0.
\end{align*}
This concludes the proof of the compact embeddings.
\end{proof}
A consequence of Theorem \ref{teo.GaffIneqcontman} is contained in the next theorem, and will be useful in Section \ref{sec:poincineq}.
Let us define $$\mathcal{L}^{1,p}(M, E_0^h)=\left\{\alpha\in L^p(M, E_0^h):\, \d^M_C \alpha\in L^{p}(M, E_0^{h+1}),\, \, \delta^M_C \alpha \in L^{p}(M, E_0^{h-1})\right\}$$ when $h\neq n,n+1$ endowed with the norm
$$\|\alpha\|_{L^p(M,E_0^h)}+\|\d_C^M\alpha\|_{L^{p}(M,E_0^{h+1})}+\|\delta^M_C\alpha\|_{L^{p}(M,E_0^{h-1})}.$$

On the other hand, 
if $h=n$, we define the space 
$$\mathcal{L}_n^{2,p}(M, E_0^n)=\left\{\alpha\in L^p(M, E_0^n):\ \d^M_C \alpha\in L^{p}(M, E_0^{n+1}), \ \ \d^M_C\delta^M_C \alpha \in L^{p}(M, E_0^{n})\right\}$$ equipped with the norm  
$$\|\alpha\|_{L^p(M,E_0^n)}+\|\d_C^M\alpha\|_{L^{p}(M,E_0^{n+1})}+\|\d^M_C\delta^M_C\alpha\|_{L^{p}(M,E_0^{n})}.$$ 

Similarly, if 
 $h=n+1$, we set 
$$\mathcal{L}_{n+1}^{2,p}(M, E_0^{n+1})=\left\{\alpha\in L^p(M, E_0^{n+1}):\ \delta^M_C\d_C^M\alpha\in L^{p}(M, E_0^{n+1}), \ \ \delta^M_C \alpha \in L^{p}(M, E_0^{n})\right\}$$ equipped with the norm  
$$\|\alpha\|_{L^p(M,E_0^{n+1})}+\|\delta_C^M\alpha\|_{L^{p}(M,E_0^{n})}+\|\delta^M_C\d_C^M\alpha\|_{L^{p}(M,E_0^{n+1})}.$$ 

\begin{teo}\label{spazi uguali}
Let $(M,H,g^M)$ be a sub-Riemannian contact compact manifold without boundary. Let $1\leq h\leq 2n$, $1< p<\infty$. With the notation above, 
\begin{itemize}
	\item[i)] if $h\neq n, n+1$,  $\mathcal{L}^{1,p}(M, E_0^h)=W^{1,p}(M, E_0^h)$,
	\item[ii)] if $h= n$,  $\mathcal{L}_n^{2,p}(M, E_0^n)=W^{2,p}(M, E_0^n)$,
	\item[iii)] if $h= n+1$,  $\mathcal{L}_{n+1}^{2,p}(M, E_0^{n+1})=W^{2,p}(M, E_0^{n+1})$.
\end{itemize}
\end{teo}
The proof of this result is contained in Appendix B.
%
%
%
%
%
%
%
%
%
%
\subsection{A global maximal hypoellipticity  on contact manifolds}\label{subsec:Rumlapmangeom}

We end this introductory section  proving a global $L^p$-maximal hypoellipticity of {Rumin Laplacians}, for any $p\in ]1,\infty[$ when $M$ is compact. 
 The result is stated in  Theorem \ref{teo.globalmaxhypocont}.

As in the Heisenberg case, the operator $\Delta_{M,h}$ (see \eqref{defi.RuminLapcontman intro}) is a self-adjoint operator that commutes with the Hodge operator as
$\ast \Delta_{M,h}=\Delta_{M,2n+1-h}\ast$,
and the Laplacians commute with $\d^M_C$ and $\delta^M_C$. 


As already mentioned, Rumin in \cite{rumin_jdg} (Theorem p.\,$190$)  proved that these Laplacians are $L^2$-{maximal hypoelliptic operators} and hypoelliptic (see also Theorem $3.1$ in \cite{rumin_gafa} and \cite{RS}). We now aim to prove a \textit{global} maximal hypoellipticity in $L^p$, as stated in Theorem \ref{teo.globalmaxhypocont} below. The proof uses the local result  \eqref{eq:lp-hypo} and, by reasoning in charts,  we have to address the issue of lack of commutation between pull-back and $\delta^M_C$. Indeed, since $\delta_C^M$ involves the Hodge operator, $\delta_C^\H\psi^\#\neq\psi^\#\delta_C^M$. In \cite{BTT},  an explicit expression for the extent to which  the co-differential and the pull-back fail to commute is given. In particular, keeping in mind Theorem 3.10 in \cite{BTT},
we recall the statements of Proposition $4.9$ and Remark $4.10$ in \cite{BTT}, that read as follows.
\begin{prop}\label{prop.notcommdeltapull}
Let $(M,H,g^M)$ be a sub-Riemannian contact manifold with $C^k$-bounded geometry, $k\ge 3$. For any $a\in M$ there exists a contact map $\psi_a$ satisfying the conditions of Definition \ref{defi.subRiemcontmanwboundgeom} and such that
the map 
			\begin{align*}
				(d\psi_a)_e\colon T_e\mathbb{H}^n\to T_aM\,,\\
				v\mapsto (d\psi_a)_e(v)\,,
			\end{align*}
			is an orthogonal linear map. 
			
			Let us define a differential operator $\mathcal{P}$ on $\H^n$ as
\begin{equation}\label{eq.Pcommnondcstar}
\mathcal{P}=\mathcal{P}(\psi_a^\#\alpha):=\delta^\H_C\psi_a^\#\alpha-\psi_a^\#\delta_C^M\alpha,\quad\quad\text{for every $\alpha\in C^\infty(M,E_0^h)$.}
\end{equation}
Let now $1<p<\infty$, $m\in\N$. For any $\e>0$ there exists $\tilde\eta=\tilde\eta_\e>0$ such that if $\eta<\tilde\eta$ and $\alpha\in C^\infty(M,E_0^h)$ supported in $B_M(a,\eta)$, we have
\begin{equation*}
\|\mathcal{P}(\psi_a^\#\alpha)\|_{W^{m,p}(A,E_0^{h-1})}\leq \e\,\|\psi_a^\#\alpha\|_{W^{m+1,p}(A,E_0^h)}\quad\quad\text{if $h\neq n+1$}
\end{equation*}
and
\begin{equation*}
\|\mathcal{P}(\psi_a^\#\alpha)\|_{W^{m,p}(A,E_0^{n})}\leq \e\,\|\psi_a^\#\alpha\|_{W^{m+2,p}(A,E_0^{n+1})}\quad\quad\text{if $h=n+1$},
\end{equation*}
where $A:=\psi_a^{-1}(B_M(a,\eta))$.
\end{prop}%

\begin{oss}\label{oss.stimecommdeltadC}
Let $1<p<\infty$ and $a\in M$. Consider a countable, locally finite, atlas $\mc U:=\{ B(a_j, \eta), \psi_{j}\}$  (we refer to Remark 3.5 in \cite{BTT} for details), where $\psi_j: B(e,\eta/C_M)\to M$. 
 Let now $\{\chi_j\}$ be a partition of unity {  subordinate} to the atlas. 
 Combining Lemma $4.1$ in \cite{BFP3} and Lemma $4.11$ in \cite{BTT}, and going trought the proof of Theorem 5.4 in \cite{BTT}, if $\alpha\in C^\infty(M,E_0^h)$ then the $L^p$-norms of the commutators $[\d_C^M,\chi_j]\alpha, [\delta_C^M,\chi_j]\alpha, [\delta_C^M\d_C^M,\chi_j]\alpha, [\d_C^M\delta_C^M,\chi_j]\alpha$ can be estimated, up to a constant, by $\sum_{k\in I_j}\|\psi_k^\#(\chi_k\alpha)\|_{W^{s-1,p}(\H^n,E_0^h)}$, where $s$ is the order of $\d_C^M, \delta_C^M, \delta_C^M\d_C^M, \d_C^M\delta_C^M$, respectively (again with the identification of the $L^p$ norm with $W^{0,p}$ when $s=1$).
\end{oss}
If we assume $M$ is compact we can  prove a $L^p$-maximal hypoellipticity estimate. 
\begin{teo}\label{teo.globalmaxhypocont}
Let $(M,H,g^M)$ be a sub-Riemannian compact contact manifold without boundary. Let $0\leq h\leq 2n+1$, and $1< p<\infty$. Then, there exists a geometric constant $C_{M,p}>0$ such that, for every $\alpha\in W^{2\ell,p}(M, E_0^h)$, we have
\begin{equation*}
\|\alpha\|_{W^{2\ell,p}(M,E_0^h)}\leq C\Big(\|\Delta_{M,h}\alpha\|_{L^p(M,E_0^h)}+\|\alpha\|_{L^p(M,E_0^h)}\Big),
\end{equation*}
where $\ell=1$ if $h \neq n,n+1$ and $\ell=2$ otherwise.
\end{teo}
\begin{proof}
The argument that we adopt in this  proof is akin to the proof of Theorem 5.4 in \cite{BTT} but, since $M$ is compact we can avoid some technicality; therefore we limit ourselves to proving the result for  $h\neq n,n+1$, hence $\ell=1$.

Let $\alpha\in C^\infty(M,E_0^h)$. Let $\left\{B_M(x_j,\eta), \psi_{j}\right\}_{j=1}^s$ a finite covering of M. Let $\{\chi_j\}_{j}$ be a partition of unity subordinated to the atlas. Without loss of generality, we can assume that $\psi_j^{-1}(\mathrm{supp}\,\chi_j)\subseteq B_\H(e,\eta/C_M)$. 

By the maximal hypoellipticity of $\Delta_{\H,h}$ on $B_\H(e,1)$ (see \eqref{eq:lp-hypo}), we have
\begin{gather*}
\begin{split}
\|\alpha\|^p&_{W^{2,p}(M,E_0^h)}=\sum_{j=1}^s\,\|\psi_j^\#(\chi_j\alpha)\|^p_{W^{2,p}(B_\H(e,\eta/C_M),E_0^h)}\\
&\leq C_p\sum_{j=1}^s\,\left(\|\Delta_{\H,h}\psi_j^\#(\chi_j\alpha)\|^p_{L^p(B_\H(e,\eta/C_M),E_0^h)}+\|\psi_j^\#(\chi_j\alpha)\|^p_{L^p(B_\H(e,\eta/C_M),E_0^h)}\right).
\end{split}
\end{gather*}

For every $j$ fixed, using Proposition \ref{prop.notcommdeltapull}, we commute $\Delta_{\H,h}$ with $\psi^\#_j$, obtaining
\begin{align*}
\|\Delta_{\H,h}(\psi_j^\#(\chi_j\alpha))\|^p_{L^p_\H}&=\|(\d_C^\H\delta_C^\H+\delta_C^\H\d_C^\H)(\psi_j^\#(\chi_j\alpha))\|^p_{L^p_\H}\\
&=\|\d_C^\H\delta_C^\H(\psi_j^\#(\chi_k\alpha))+\delta_C^\H\psi_j^\#(\d_C^M(\chi_j\alpha))\|^p_{L^p_\H} \\
&=\|\d_C^\H(\psi_j^\#(\delta_C^M(\chi_j\alpha)))+\d_C^\H \mc P(\psi_j^\#(\chi_j\alpha))\\
&\quad\quad\quad\quad+\psi_j^\#(\delta_C^M\d_C^M(\chi_j\alpha))+\mc P(\psi_j^\#(\d_C^M(\chi_j\alpha)))\|^p_{L^p_\H} \\
&=\|\psi_j^\#(\d_C^M\delta_C^M(\chi_j\alpha))+\psi_j^\#(\delta_C^M\d_C^M(\chi_j\alpha))\\
&\quad\quad\quad\quad+\d_C^\H \mc P(\psi_j^\#(\chi_j\alpha))+\mc P(\psi_j^\#(\d_C^M(\chi_j\alpha)))\|^p_{L^p_\H} \\
&\leq c_p\Big({\|\psi_j^\#(\Delta_{M,h}(\chi_j\alpha))\|^p_{L^p_\H}}\\
&\quad\quad\quad\quad+{\|\d_C^\H \mc P(\psi_j^\#(\chi_j\alpha))\|^p_{L^p_\H}+\|\mc P(\psi_j^\#(\d_C^M(\chi_j\alpha)))\|^p_{L^p_\H}}\Big)=:S_1+S_2.
\end{align*}
Now, if $0<\e<1$, applying Proposition \ref{prop.notcommdeltapull} to $S_2$, for every $\eta<\tilde\eta_\e$, we have
\begin{align*}
\|\d_C^\H \mc P(\psi_j^\#(\chi_j\alpha))\|^p_{L^p(B_\H(e,\eta/C_M),E_0^h)}&+\|\mc P(\psi_j^\#(\d_C^M(\chi_j\alpha)))\|^p_{L^p(B_\H(e,\eta/C_M),E_0^h)}\\
&\quad\quad\quad\leq\e^p c'_p\,\|\psi_j^\#(\chi_j\alpha)\|^p_{W^{2,p}(B_\H(e,\eta/C_M),E_0^h)},
\end{align*}where the constant $c_p'$ on the right hand side does not depend on $\eta$.
Using the estimates in Remark \ref{oss.stimecommdeltadC}, we can estimate the remaining term $S_1$ as
\begin{gather}\label{eq.timedeltaMast}
\begin{split}
\|\psi_j^\#(\Delta_{M,h}&(\chi_j\alpha))\|^p_{L^p(B_\H(e,\eta/C_M),E_0^h)}\leq \|\psi_j^\#(\chi_j\Delta_{M,h}\alpha)\|^p_{L^p(B_\H(e,\eta/C_M),E_0^h)}\\
&\quad+\|\psi_j^\#[\d_C^M\delta_C^M,\chi_j]\alpha\|^p_{L^p(B_\H(e,\eta/C_M),E_0^h)}+
\|\psi_j^\#[\delta_C^M\d_C^M,\chi_j]\alpha\|^p_{L^p(B_\H(e,\eta/C_M),E_0^h)}\\
&\leq\|\psi_j^\#(\chi_j\Delta_{M,h}\alpha)\|^p_{L^p(B_\H(e,\eta/C_M),E_0^h)}+c_{p,\eta }\,\sum_{j=1}^s\|\psi_j^\#(\chi_j\alpha)\|^p_{W^{1,p}(B_\H(e,\eta/C_M),E_0^h)},
\end{split}
\end{gather}where $c_{p,\eta}$ depends on $\eta$.
By Ehrling's inequality in $B_\H(e,\eta/C_M)$, if $0<\sigma<1$  (see Theorem \ref{teo.compactembHn}),  there exists a constant $c_\sigma>0$ such that
\begin{gather*}
\begin{split}
\|\psi_j^\#(\Delta_{M,h}&(\chi_j\alpha))\|^p_{L^p(B_\H(e,\eta/C_M),E_0^h)}\leq \,\|\psi_j^\#(\chi_j\Delta_{M,h}\alpha)\|^p_{L^p(B_\H(e,\eta/C_M),E_0^h)}
\\
&+c_{p,\eta}\,\sum_{j=1}^s\Big(\sigma^p\|\psi_j^\#(\chi_j\alpha)\|^p_{W^{2,p}(B_\H(e,\eta/C_M),E_0^h)}
+c^p_\sigma\|\psi_j^\#(\chi_j\alpha)\|^p_{L^p(B_\H(e,\eta/C_M),E_0^h)}\Big).
\end{split}
\end{gather*}
Putting together all the computations and recalling that $\psi_j^{-1}(\mathrm{supp}\,\chi_j)\subseteq B_\H(e,\eta/C_M)\subseteq B_\H(e,1)$, up to relabelling the constants, we have obtained that
\begin{gather*}
\begin{split}
\|\alpha\|^p_{W^{2,p}(M,E_0^h)}&\leq \,\sum_{j=1}^s\|\psi_j^\#(\chi_j\Delta_{M,h}\alpha)\|^p_{L^p(B_\H(e,1),E_0^h)}+\e^p c'_p\,\sum_{j=1}^s\|\psi_j^\#(\chi_j\alpha)\|^p_{W^{2,p}(B_\H(e,1),E_0^h)}\\
&\quad+c_{p, \eta}\,\sum_{j=1}^s\left(\sigma^p\|\psi_j^\#(\chi_j\alpha)\|^p_{W^{2,p}(B_\H(e,1),E_0^h)}+c^p_\sigma\|\psi_j^\#(\chi_j\alpha)\|^p_{L^{p}(B_\H(e,1),E_0^h)}\right)\\
&= \|\Delta_{M,h}\alpha\|^p_{L^p(M,E_0^h)}+\e^p c'_p\|\alpha\|^p_{W^{2,p}(M,E_0^h)}\\
&\quad\quad+c_{p, \eta}\,\left(\sigma^p\|\alpha\|^p_{W^{2,p}(M,E_0^h)}+c^p_\sigma\|\alpha\|^p_{L^p(M,E_0^h)}\right).
\end{split}
\end{gather*}
We choose $\e^p c'_p<1$ (recall that $c_p'$ is independent of $\eta$). Once we have fixed $\e$, we can fix also $\eta<\eta_\e$ and we can choose $\sigma<1$ such that $c_{p, \eta}\sigma^p<1$. In this way, we can reabsorb  terms with the $W^{2,p}$-norm of the last inequality, completing the proof.
\end{proof}



\section{$L^2$-Hodge decomposition theorems on compact contact manifolds}\label{HodgeL2}
In the sequel we always assume $M$ to be a compact contact manifold (again without boundary as above) and $1\le h\le 2n$.

Our proof of the $L^p$-{Hodge decomposition} pass preliminarly through an $L^2$-{Hodge decomposition}.
In this section, at first, we prove the existence and uniqueness of solutions of a potential of a form belonging to $L^2$ and orthogonal to the space of harmonic forms. Thanks to the global maximal hypoellipticity of $\Delta_M$, proved in  Theorem \ref{teo.globalmaxhypocont} we are able to obtain an $L^2$-{Hodge decomposition} involving Sobolev spaces. The proofs of some statements are postponed and we adress the interested reader to Appendix A.


From now on, for sake of simplicity, we denote the Rumin differential  and  co-differential respectively by $\d_C$ and $\delta_C$, dropping the superscript $M$.


\begin{defi}\label{defi.harmRform}
A form $\alpha\in L^1(M,E_0^h)$ is called \textit{harmonic} if $\Delta_{M,h}\alpha=0$ in the sense of distribution. 
\end{defi}
We define  the space of all $h$-harmonic forms as $\mathcal{H}^h:=\ker\,\Delta_{M,h}$ in the sense of distribution.

\begin{oss}\label{oss.hypoLaprum}
We observe that $\alpha\in L^1(M,E_0^h)$ is harmonic  if and only if $\alpha\in C^\infty(M,E_0^h)$ and $\alpha$ is harmonic in the usual sense, hence  $\mathcal{H}^h\subset C^\infty(M, E_0^h)$.
\end{oss}

For every $\alpha\in L^p(M,E_0^h), \beta\in L^{p'}(M,E_0^h)$, where $\frac{1}{p}+\frac{1}{p'}=1$ and $1<p<\infty$, we denote the inner product on $L^2(M,E_0^h)$, as usual, by
$$(\alpha,\beta):=\int_M\langle\alpha,\beta\rangle\,\d V=\int_M\alpha\wedge\ast\beta.$$
We can now give a characterization of the harmonicity of a form.
\begin{prop}\label{prop.characharmRum}
A form $\alpha\in L^1(M,E_0^h)$ is harmonic if and only if $\d_C\alpha=\delta_C\alpha=0$. Moreover, 
if $\alpha$ is an $n$-form, then the harmonicity of $\alpha$ is equivalent to $\d_C\delta_C\alpha=\d_C\alpha=0$.
On the other hand, if $\alpha$ is an $(n+1)$-form, then the harmonicity of $\alpha$ is equivalent to $\delta_C\d_C\alpha=\delta_C\alpha=0$.
\end{prop}
\begin{proof} By the hypoellipticity of $\Delta_M$, we can assume that $\alpha\in C^\infty(M,E_0^h)$.
The case for $h\neq n, n+1$ is similar to the Riemannian case (see Lemma $3.3.5$ in \cite{jost}). Let us now prove the thesis for $h=n$.

If $\d_C\alpha=\delta_C\alpha=0$ or $\d_C\delta_C\alpha=\d_C\alpha=0$, then $\alpha$ is harmonic by  definition of $\Delta_M$. On the contrary, we have
\begin{align*}
(\Delta_M\alpha,\alpha)&=\Big(((\d_C\delta_C)^2+\delta_C\d_C)\alpha,\alpha\Big) =(\d_C\delta_C\d_C\delta_C\alpha,\alpha)+(\delta_C\d_C\alpha,\alpha) \\
&=\|\d_C\delta_C\alpha\|^2_{L^2(M,E_0^n)}+\|\d_C\alpha\|^2_{L^2(M,E_0^{n+1})}.
\end{align*}
Now, if $\alpha$ is harmonic, i.e.\,$\Delta_M\alpha=0$, hence $\|\d_C\delta_C\alpha\|^2_{L^2(M,E_0^n)}=\|\d_C\alpha\|^2_{L^2(M,E_0^{n+1})}=0$, and this implies that $\d_C\delta_C\alpha=\d_C\alpha=0$. Thus,
$$\|\delta_C\alpha\|^2_{L^2(M,E_0^{n-1})}=(\delta_C\alpha,\delta_C\alpha)=(\,\d_C\delta_C\alpha,\alpha)=0,$$
that implies also $\delta_C\alpha=0$. The case $h=n+1$ can be handled similarly.
\end{proof}

In the sequel, we shall
 write $C^\infty(M,E_0^{h+1})$ instead of $C_0^\infty(M,E_0^{h+1})$, since $M$ is compact and without boundary, and thus we do not need to restrict our considerations to compactly supported forms.
Occasionally, we shall also adopt the notation $W^{\ell,p}(M,E_0^h)$ without specifying each time that we are considering $\ell=1$ if $h\neq n,n+1$ and $\ell=2$ if $h=n$ or $n+1$. 

\medskip

 Since $M$ is compact, we will prove that $\mathcal{H}^h$ has finite dimension, for every $h$, see Proposition \ref{teo.findimkerlap} below. We begin with a remark.

\begin{prop}\label{prop.LpconvHharm}
For every $h$, $\mathcal{H}^h$ is closed with respect to the $L^p$ (weak) convergence.
\end{prop}
\begin{proof}
If $(\alpha_k)_{k\in\N}\subseteq\mathcal{H}^h$ is a sequence that converges (weakly) to $\alpha$ in $L^p(M,E_0^h)$, then, by Proposition \ref{prop.characharmRum}, we have
\begin{align*}
0=(\d_C\alpha_k,\varphi)=(\alpha_k,\delta_C\varphi)\to(\alpha,\delta_C\varphi),\quad\quad\forall\,\varphi\in C^\infty(M,E_0^{h+1}), \\
0=(\delta_C\alpha_k,\psi)=(\alpha_k,\d_C\psi)\to(\alpha,\d_C\psi),\quad\quad\forall\,\psi\in C^\infty(M,E_0^{h-1}).
\end{align*}
Hence $\Delta_{M,h}\alpha=0$ in distributional sense, that, by Remark \ref{oss.hypoLaprum}, implies $\alpha\in\mathcal{H}^h$.
\end{proof}

The following results extends to our setting a result stated in \cite{morrey} (see Theorem $7.3.1$ therein).
\begin{prop}\label{teo.findimkerlap}
$\mathcal{H}^\bullet$ has finite dimension.
\end{prop}
The proof can be found in the Appendix A.

For any $h$, we  define the orthogonal complement of harmonic space as
$$\left(\mathcal{H}^h\right)^\bot:=\left\{\alpha\in L^1(M,E_0^h)\,\,|\,\,(\alpha,\beta)=0,\quad\text{for all $\beta\in\mathcal{H}^h$}\right\}.$$

Since $\mathcal{H}^h$ is finite-dimensional, the orthogonal $L^2$-decomposition follows:
$$L^2(M,E_0^h)=\mathcal{H}^h\oplus\left(\left(\mathcal{H}^h\right)^\bot\cap L^2(M,E_0^h)\right)\,.$$
 We call  the orthogonal projection of $L^2$ over $\mathcal{H}^h$, 
$$H:L^2(M,E_0^h)\longrightarrow \mathcal{H}^h,$$
the harmonic projection (which is unique). 

Given a  form in $\left(\mathcal{H}^\bullet\right)^\bot\cap L^2(M,E_0^\bullet)$ we prove the existence of a potential adopting the approach used by Morrey in \cite{morrey}. As a corollary,  we recover the Hodge decomposition for smooth form stated by Rumin in \cite{rumin_jdg}, Corollary p.290.
\begin{teo}\label{teo.733morreyexistence}
Let $\alpha\in L^2(M,E_0^h)$ such that $\alpha$ is orthogonal to $\mathcal{H}^h$. Thus,
\begin{enumerate}
\item if $h\neq n,n+1$, there exists a unique form $\gamma\in W^{1,2}(M,E_0^h)\cap\left(\mathcal{H}^h\right)^\bot$ such that
\begin{equation}\label{eq.soldeboleLaphodgegen}
(\d_C\gamma,\d_C\zeta)+(\delta_C\gamma,\delta_C\zeta)=(\alpha,\zeta),\quad\quad\forall\,\zeta\in W^{1,2}(M,E_0^h);
\end{equation}
\item if $h=n$, there exists a unique form $\gamma\in W^{2,2}(M,E_0^n)\cap\left(\mathcal{H}^n\right)^\bot$ such that
\begin{equation}\label{eq.soldeboleLaphodgen}
(\d_C\delta_C\gamma,\d_C\delta_C\zeta)+(\d_C\gamma,\d_C\zeta)=(\alpha,\zeta),\quad\quad\forall\,\zeta\in W^{2,2}(M,E_0^n);
\end{equation}
\item if $h=n+1$, there exists a unique form $\gamma\in W^{2,2}(M,E_0^{n+1})\cap\left(\mathcal{H}^{n+1}\right)^\bot$ such that
\begin{equation}\label{eq.soldeboleLaphodgenp1}
(\delta_C\d_C\gamma,\delta_C\d_C\zeta)+(\delta_C\gamma,\delta_C\zeta)=(\alpha,\zeta),\quad\quad\forall\,\zeta\in W^{2,2}(M,E_0^{n+1}).
\end{equation}
\end{enumerate}

\end{teo}
The proof of this result is contained in the Appendix A.
The form $\gamma$ in the previous theorem, which is uniquely determined by $\alpha$, is called the \textit{potential of $\alpha$}.
By the hypoellipticity of $\Delta_M$, we have:
\begin{coro}\label{coro.esistsolCinf}
Let $\alpha\in C^\infty(M,E_0^h)$ such that $\alpha$ is orthogonal to $\mathcal{H}^h$. Then, there exists a unique form $\gamma\in C^\infty(M,E_0^h)\cap\left(\mathcal{H}^h\right)^\bot$ such that $\Delta_{M, h}\gamma=\alpha$.
\end{coro}

As shown in Appendix A (see Remark \ref{oss.contweaksolL2-app}), we can observe:
\begin{oss}\label{oss.contweaksolL2}  
the operator, that for every $\alpha\in L^2(M,E_0^h)\cap\left(\mathcal{H}^h\right)^\bot$ associates $\gamma\in W^{\ell,2}(M,E_0^h)\cap\left(\mathcal{H}^h\right)^\bot$, is a continuous operator and   we have
\begin{equation}\label{eq.contL2weksoll2}
\|\gamma\|_{W^{\ell,2}(M,E_0^h)}\leq\frac{2}{C}\|\alpha\|_{L^2(M,E_0^h)},\quad\quad\forall\,\alpha\in L^2(M,E_0^h)\cap\left(\mathcal{H}^h\right)^\bot,
\end{equation}
where  $\ell=1$ if $h\neq n,n+1$ and $\ell=2$ for $h= n,n+1$
\end{oss}
Now we prove the following estimates for smooth forms.
\begin{oss}\label{passo preliminare}
If $\gamma\in C^\infty(M,E_0^h)\cap\left(\mathcal{H}^h\right)^\bot$ is the solution of $\Delta_{M,h}\gamma=\alpha\in C^\infty(M,E_0^h)\cap\left(\mathcal{H}^h\right)^\bot$, then 
\begin{equation}\label{eq.stimaCinfweak}
\|\gamma\|_{W^{2\ell,2}(M,E_0^h)}\leq C\|\alpha\|_{L^2(M,E_0^h)},
\end{equation}
where $\ell=2$ if $h=n,n+1$, and $\ell=1$ otherwise.
\end{oss}
\begin{proof}
By Theorem \ref{teo.globalmaxhypocont}, we have
\begin{align*}
\|\gamma\|_{W^{2\ell,2}}\leq C\Big(\|\Delta_{M,h}\gamma\|_{L^2}+\|\gamma\|_{L^2}\Big)=C\Big(\|\alpha\|_{L^2}+\|\gamma\|_{L^2}\Big)\leq C\Big(\|\alpha\|_{L^2}+\|\gamma\|_{W^{\ell,2}}\Big).
\end{align*}
Using the inequality \eqref{eq.contL2weksoll2}, we obtain
$
\|\gamma\|_{W^{2\ell,2}}\leq C\|\alpha\|_{L^2},
$
and we have done.
\end{proof}

\subsection{Hodge Theorems for $C^\infty$-Rumin forms}\label{sec:finalcap}
We define the following operator.
\begin{defi}\label{defi.GreenopCinf}
We define the operator
$$G:C^\infty(M,E_0^h)\longrightarrow C^\infty(M,E_0^h)\cap\left(\mathcal{H}^h\right)^\bot,$$
such that $G(\alpha)$ is the unique solution of the $\Delta_{M,h}G(\alpha)=\alpha-H(\alpha)$ as in Corollary \ref{coro.esistsolCinf}. 
\end{defi}
We can refer to $G$ as the Green's operator associated to $\Delta_M$ (see e.g. \cite{warner}).

As noticed by Rumin (see Corollary p. 290 in \cite{rumin_jdg}), the Laplacian on $M$ induces an orthogonal decomposition on smooth compactly supported forms, and the cohomology of $M$ is represented by harmonic forms. Here, we prove a slightly general decomposition.
\begin{coro}[see \cite{rumin_jdg}]\label{teo.HodgedecthCinf}
We have the decomposition 
\begin{equation}\label{eq.decompsmoothform}
C^\infty(M,E_0^h)=\mathcal{H}^h\oplus\d_C\Big(C^\infty(M,E_0^{h-1})\Big)\oplus\delta_C\Big(C^\infty(M,E_0^{h+1})\Big).
\end{equation}
Moreover, we can write
\begin{equation}\label{eq.decompCinfGhodge}
C^\infty(M,E_0^h)=\mathcal{H}^h\oplus\left\{
\begin{array}{ll}
\d_C\delta_C G\left(C^\infty(M,E_0^h)\right)\oplus\delta_C\d_C G\left(C^\infty(M,E_0^h)\right),&\text{if $h\neq n,n+1$},\\
\d_C\delta_C\d_C\delta_C G\left(C^\infty(M,E_0^h)\right)\oplus\delta_C\d_C G\left(C^\infty(M,E_0^h)\right),&\text{if $h=n$},\\
\d_C\delta_C G\left(C^\infty(M,E_0^h)\right)\oplus\delta_C\d_C\delta_C\d_C G\left(C^\infty(M,E_0^h)\right),&\text{if $h=n+1$}.
\end{array}
\right. 
\end{equation}
\end{coro}
%
%
\begin{proof}
Let $\alpha\in C^\infty(M,E_0^h)$, let $H(\alpha)$ be the harmonic projection of $\alpha$ over $\mathcal{H}^h$ and let $G(\alpha)\in C^\infty(M,E_0^h)\cap\left(\mathcal{H}^h\right)^\bot$ be as in Definition \ref{defi.GreenopCinf}.
Using the explicit expression of $\Delta_M$ given in \eqref{defi.RuminLapcontman intro}, we easily infer that $\alpha=H(\alpha)+\Delta_{M,h}G(\alpha)$ can be written as
$$
\alpha=H(\alpha)+\delta_C\zeta+\d_C\omega,
$$
where  $\zeta\in C^\infty(M,E_0^{h+1})$ and $\omega\in C^\infty(M,E_0^{h-1})$ are defined as follow: 
$$\zeta=\d_CG(\alpha)\quad \text{for $h\neq n+1$, whereas} \quad \zeta=\d_C\delta_C\d_CG(\alpha)\ \text{for $h=n+1$;} $$
and similarly,  
 $$\omega=\delta_CG(\alpha)\quad \text{for  $h\neq n,$ whereas}\quad \ \omega=\delta_C\d_C\delta_CG(\alpha)\ \text{for $h=n$}.$$
 Thus, we have the decomposition  \eqref{eq.decompsmoothform} which can be summarized in the following expression
\begin{equation}\label{eq.splitformecinfty}
\alpha=H(\alpha)+
\left\{
\begin{array}{ll}
\delta_C{\left(\d_CG(\alpha)\right)}+\d_C{\left(\delta_CG(\alpha)\right)},&\text{if $h\neq n,n+1$},\\
\delta_C{\left(\d_CG(\alpha)\right)}+\d_C{\left(\delta_C\d_C\delta_CG(\alpha)\right)},&\text{if $h=n$},\\
\delta_C{\left(\d_C\delta_C\d_CG(\alpha)\right)}+\d_C{\left(\delta_CG(\alpha)\right)},&\text{if $h=n+1$},
\end{array}
\right. 
\end{equation}
We want to prove the orthogonality in \eqref{eq.decompsmoothform}. By Proposition \ref{prop.characharmRum}, keeping in mind that $\d_CH(\alpha)=\delta_C H(\alpha)=0$, we have
\begin{align*}
(H(\alpha)&,\delta_C\zeta)=({\d_CH(\alpha)},\zeta)=0,\quad (H(\alpha),\d_C\omega)=({\delta_CH(\alpha)},\omega)=0, \\
&\text{and}\quad\quad\quad(\delta_C\zeta,\d_C\omega)=(\zeta,\d_C^2\omega)=0.
\end{align*}
Hence, $\d_C\Big(C^\infty(M,E_0^{h-1})\Big), \delta_C\Big(C^\infty(M,E_0^{h+1})\Big)$ and $\mathcal{H}^h$ are orthogonal to each other, completing the proof.
\end{proof}

\medskip

\subsection{Green operator and Hodge theorems for $L^2$-Rumin forms involving Sobolev spaces}\label{sec:finalcapL2}
Now we generalize the decompositions of smooth forms stated in Corollary \ref{teo.HodgedecthCinf} to the $L^2$-forms, in the same spirit of  the decompostion given in formula (1.2) by \cite{ISS}. 

The next statement contains a regularity result. Indeed, thanks to Theorem \ref{teo.globalmaxhypocont}, we show that the weak solutions obtained in Theorem \ref{teo.733morreyexistence} belong to $W^{2\ell,2}(M,E_0^h)$ (again, $\ell=2$ if $h=n,n+1$, and $\ell=1$ otherwise).  

\begin{teo}\label{teo.greenopregsol}Let $\ell=2$ if $h=n,n+1$, and $\ell=1$ otherwise.
If $\alpha\in L^2(M,E_0^h)\cap\left(\mathcal{H}^h\right)^\bot$ and $\gamma\in W^{\ell,2}(M,E_0^h)\cap\left(\mathcal{H}^h\right)^\bot$ is the weak solution  as in Theorem \ref{teo.733morreyexistence}, then $\gamma\in W^{2\ell,2}(M,E_0^h)\cap\left(\mathcal{H}^h\right)^\bot$ and 
\begin{equation}\label{32 bis}
\|\gamma\|_{W^{2\ell,2}}\leq C\|\alpha\|_{L^2}\,.
\end{equation}
\end{teo}
\begin{proof}

 Let $(\alpha_k)_{k\in\N}\subseteq C^\infty(M,E_0^h)\cap\left(\mathcal{H}^h\right)^\bot$ such that $\alpha_k\to \alpha$ in $L^2$ and let $\gamma_k\in C^\infty(M,E_0^h)\cap\left(\mathcal{H}^h\right)^\bot$ be the unique solution of $\Delta_{M,h}\gamma_k=\alpha_k$ as in Corollary \ref{coro.esistsolCinf}, for every $k$. Then, by \eqref{eq.stimaCinfweak} applied to $\gamma_k$ and using the fact that $\alpha_k$ converges to $\alpha$ in $L^2$, we have
$$\|\gamma_k\|_{W^{2\ell,2}}\leq C\|\alpha_k\|_{L^2}\leq \|\alpha\|_{L^2}+1,\quad\quad\forall\,k\gg 1,$$
showing that $(\gamma_k)_{k\in\N}$ is a bounded sequence in $W^{2\ell,2}(M,E_0^h)\cap\left(\mathcal{H}^h\right)^\bot$. Thus, by Remark \ref{teo.compactembcontact}, up to pass to a subsequence and recalling that $\mathcal{H}^h$ is closed under $L^2$ convergence (see Proposition \ref{prop.LpconvHharm}), there exists $\gamma\in W^{2\ell,2}(M,E_0^h)\cap\left(\mathcal{H}^h\right)^\bot$ such that $\gamma_k\to \gamma$ in $W^{\ell,2}$ and $\gamma_k\rightharpoonup\gamma$ in $W^{2\ell,2}$.
We prove now that $\gamma$ is the weak solution of $\Delta_{M,h}\gamma=\alpha$, where this equality is meant  in the sense of Theorem \ref{teo.733morreyexistence}. 

Indeed, for example if $h\neq n,n+1$ (hence $\ell=1$), the differential operators $\d_C$ and $\delta_C$ have order $1$ that, by the strong convergence $\gamma_k\to\gamma$ in $W^{1,2}$, we have
$$\d_C\gamma_k\to\d_C\gamma\quad\text{and}\quad\delta_C\gamma_k\to\delta_C\gamma\quad\text{in $L^2$}$$
and
\begin{align*}
(\d_C\gamma,\d_C\zeta)+(\delta_C\gamma,\delta_C\zeta)&=\lim_{k\to+\infty}\,\Big((\d_C\gamma_k,\d_C\zeta)+(\delta_C\gamma_k,\delta_C\zeta)\Big)=\lim_{k\to+\infty}\,(\Delta_{M,h}\gamma_k,\zeta) \\
&=\lim_{k\to+\infty}\,(\alpha_k,\zeta)=(\alpha,\zeta),\quad\quad\forall\,\zeta\in W^{1,2}(M,E_0^h).
\end{align*}
Moreover, by construction, $\gamma\in W^{2,2}(M,E_0^h)$, showing the regularity of the weak solutions of Poisson's equation. Arguing analogously, we obtain the regularity for the cases $n, n+1$.

To prove that \eqref{32 bis} holds, we can argue as in the proof of Remark \ref{passo preliminare} since, 
 by a density argument, the estimate of maximal hypoellipticity in Theorem \ref{teo.globalmaxhypocont} is true also for forms in $W^{2\ell,2}(M,E_0^h)$. 
\end{proof}
Analogously to the case of smooth forms, we can define the  operator on $L^2$-forms as follow (see \cite{scott}).
\begin{defi}\label{defi.GreenopL2case}
We define the \textit{Green's operator}
$$G:L^2(M,E_0^h)\longrightarrow W^{2\ell,2}(M,E_0^h)\cap\left(\mathcal{H}^h\right)^\bot$$
such that $G(\alpha)$ is the unique solution of the Poisson's equation 
$$\Delta_{M,h}G(\alpha)=\alpha-H(\alpha)$$ where, again,  the last equality is meant 
as in Theorem \ref{teo.733morreyexistence}. We observe that $G$ is a linear continuous operator due to the estimate \eqref{eq.stimaCinfweak} in Theorem \ref{teo.greenopregsol}.
\end{defi}
The following lemma allows to obtain the uniqueness of the decomposition of $L^2$ forms in Theorem \ref{teo.HodgedecthL2Cinf} (see Lemma 7.4.1 in \cite{morrey} and also Lemma 6.3 in \cite{scott}).
\begin{lemma}\label{lemma.orthLpforms}
 We consider $s=2$ if $h=n+1$, and $s=1$ otherwise. In addition let $r=2$ if $h=n$ and $r=1$ otherwise. If $\alpha\in W^{s,2}(M,E_0^{h-1})$, $\beta\in W^{r,2}(M,E_0^{h+1})$ and $\xi\in\mathcal{H}^h$ such that 
\begin{equation}\label{uf}\d_C\alpha+\delta_C\beta+\xi=0\,,\end{equation}  then $\d_C\alpha=\delta_C\beta=\xi=0$.
\end{lemma}
\begin{proof}
Let us show that $(\d_C\alpha, \varphi)=0 $ for every $\varphi\in C^\infty(M,E_0^h)$. Thanks to \eqref{eq.decompsmoothform}, we can write $\varphi=H(\varphi)+\d_C\varphi_1+\delta_C\varphi_2$, with $\varphi_1\in C^\infty(M,E_0^{h-1})$ and $\varphi_2\in C^\infty(M,E_0^{h+1})$. We know that $(\delta_C\beta, \d_C\varphi_1)=(\beta, \d_C^2\varphi_1)=0$ and $(\xi, \d_C\varphi_1)= (\delta_C\xi,\varphi_1)=0$. Therefore  $(\d_C\alpha,\d_C\varphi_1)=(\d_C\alpha+\delta_C\beta+\xi,\d_C\varphi_1)=0$, by \eqref{uf}. In addition, 
$(\d_C\alpha,\delta_C\varphi_2)=(\alpha,\delta_C^2\varphi_2)=0$. 
Moreover, 
$$(\d_C\alpha,\delta_C\varphi_2)=(\alpha,\delta_C^2\varphi_2)=0\quad\text{and}\quad (\d_C\alpha,H(\varphi))=0.$$
Thus, we have proved that $(\d_C\alpha,\varphi)=0$, for every $\varphi\in C^\infty(M,E_0^h)$.
Analogously, we can prove that $(\delta_C\beta,\varphi)=0$. Finally, again by \eqref{uf}, $(\xi, \varphi)=(\d_C\alpha+\delta_C\beta+\xi, \varphi)=0$, and  the proof is done.
\end{proof}

We are finally in position to obtain our regular-Hodge decomposition in $L^2$ (we stress that the direct sum is with respect to the orthogonality in the Hilbert space $L^2(M,E_0^h)$).
\begin{teo}\label{teo.HodgedecthL2Cinf}
Let $(M,H,g^M)$ be a sub-Riemannian compact contact manifold, without boundary. Then, 
\begin{equation}\label{eq.decompL2hodge}
L^2(M,E_0^h)=\mathcal{H}^h\oplus \d_C\Big(W^{s,2}(M,E_0^{h-1})\Big)\oplus \delta_C\Big(W^{r,2}(M,E_0^{h+1})\Big),
\end{equation}
where $s=2$ if $h=n+1$ and $s=1$ otherwise, and $r=2$ if $h=n$ and $r=1$ otherwise.
\end{teo}
\begin{proof}
Let $\alpha\in L^2(M,E_0^h)$ and $G(\alpha)\in W^{2\ell,2}(M,E_0^h)\cap\Big(\mathcal{H}^h\Big)^\bot$ as in Definition \ref{defi.GreenopL2case}.
With the same argument we used for proving \eqref{eq.decompsmoothform}, we have that
\begin{equation}\label{eq.L2splitform}\alpha=\delta_C\zeta+\d_C\omega+H(\alpha),
\end{equation}
where $\zeta\in W^{2,2}$ if $h= n$ and $\zeta\in W^{1,2}$ if $h\neq n$; similarly $\omega\in W^{2,2}$ if $h=n+1$ and $\omega\in W^{1,2}$ if $h\neq n+1$. 

We prove now the $L^2$-orthogonality of the decomposition. We know that the harmonic projection is unique. If we write $\alpha$ with two different decompositions, i.e.
$$\alpha=\delta_C\zeta+\d_C\omega+H(\alpha)=\delta_C\widetilde\zeta+\d_C\widetilde\omega+H(\alpha),$$
then, $\delta_C(\zeta-\widetilde\zeta)+\d_C(\omega-\widetilde\omega)=0$, implying by Lemma \ref{lemma.orthLpforms} that $\zeta=\widetilde \zeta$ and $\omega=\widetilde \omega$, completing the proof.
\end{proof}\bigskip

\section{Main results on $L^p$-Hodge decomposition}\label{sec:HodgedecLPformss}

Now we generalize the Hodge decomposition for $L^p$-forms, when $1<p<\infty$. 
Our main achievement is contained in the following theorem.
\begin{teo}\label{teo.HodgedecthLp}
Let $(M,H,g^M)$ be a sub-Riemannian compact contact manifold without boundary, $1\le h\le 2n$  and $1<p<\infty$. Then
\begin{align}
L^p(M,E_0^h)=&\mathcal{H}^h\oplus \d_C\Big(W^{1,p}(M,E_0^{h-1})\Big)\oplus \delta_C\Big(W^{1,p}(M,E_0^{h+1})\Big), \; \text{if } h\neq n, n+1\,, \label{eq.decompLphodge}\\
L^p(M,E_0^n)=&\mathcal{H}^n\oplus \d_C\Big(W^{1,p}(M,E_0^{n-1})\Big)\oplus \delta_C\Big(W^{2,p}(M,E_0^{n+1})\Big), \label{eq.decompLphodgedue}\\
L^p(M,E_0^{n+1})=&\mathcal{H}^{n+1}\oplus \d_C\Big(W^{2,p}(M,E_0^{n})\Big)\oplus \delta_C\Big(W^{1,p}(M,E_0^{n+2})\Big). \label{eq.decompLphodgetre}
\end{align}


\end{teo}
The direct sums are in the sense of the uniqueness of Lemma \ref{lemma.orthLpforms bis}.

The proof of this result follows an argument similar to the one used to obtain Theorem \ref{teo.HodgedecthL2Cinf} in the previous section, once we have proved  the $L^p$-regularity
of solutions of Poisson's equation when the datum $\alpha$ is an $L^p$-form. In Section 3,  we have shown that if  $\alpha 
\in L^2(M,E_0^h)\cap\left(\mathcal{H}^h\right)^\bot$ then its potential is in $ W^{2\ell,2}(M,E_0^h)\cap\left(\mathcal{H}^h\right)^\bot$. Now, we are going to show that, starting from $\alpha  
\in L^p(M,E_0^h)\cap\left(\mathcal{H}^h\right)^\bot$, its potential is in $W^{2\ell,p}(M,E_0^h)\cap\left(\mathcal{H}^h\right)^\bot$.



We begin with a result related to a projection on  $\mathcal{H}^\bullet$ for $L^p$-forms. The proof is similar to the one given in \cite{scott} (see Lemma 5.6 and Proposition 5.9 therein) since it basically relies on the  fact that $ \mathcal{H}^h$ has finite dimension for any $h$.

\begin{prop}\label{prop.harmoprojj}
Let $1<p<\infty$. For any $\alpha\in L^p(M, E_0^h)$ there exists a unique $h$-form in  $\mathcal{H}^h$, that we denote by $H(\alpha)$,  such that $(\alpha-H(\alpha), \beta)=0$ for any $\beta\in \mathcal{H}^h$.
Moreover the map $$H:L^p(M,E_0^h)\longrightarrow \mathcal{H}^h$$
 is a linear bounded operator.
\end{prop}
\begin{proof}
Since $\mathcal{H}^h$ has finite dimension, we can consider $\{\omega_1,\ldots,\omega_r\}$ as an orthonormal basis of $\mathcal{H}^h$. For every $\alpha\in L^p(M,E_0^h)$, we set
$$H(\alpha):=\sum_{i=1}^r\,(\alpha,\omega_i)\,\omega_i.$$
For every $\beta=\sum_i\beta_i\omega_i\in\mathcal{H}^h$, we have that
$$(\alpha-H(\alpha),\beta)=(\alpha,\beta)-(H(\alpha),\beta)=\sum_{i=i}^r\,(\alpha,\omega_i)\beta_i-\sum_{i=i}^r\,(\alpha,\omega_i){(\omega_i,\omega_i)}\beta_i=0,$$
since ${(\omega_i,\omega_i)}=1$.
The element $H(\alpha)$ is unique since, if $\widetilde\alpha=\sum_i\,\widetilde\alpha_i\omega_i\in\mathcal{H}^h$ is such that
$$(\alpha-\widetilde\alpha,\beta)=0,\quad\quad\forall\,\beta\in\mathcal{H}^h,$$
then, keeping also in mind that $\alpha-H(\alpha)\in \left(\mathcal{H}^h\right)^\bot$ and ${\widetilde\alpha-H(\alpha)}\in\mathcal{H}^h$, we have
\begin{align*}
0&=({\alpha-H(\alpha)},{\widetilde\alpha-H(\alpha)})-({\alpha-\widetilde\alpha},\widetilde\alpha-H(\alpha))\\
&=(\widetilde\alpha-H(\alpha),\widetilde\alpha-H(\alpha))=\|\widetilde\alpha-H(\alpha)\|^2.
\end{align*}
Hence, $\widetilde\alpha=H(\alpha)$.

The fact that $H:L^p(M,E_0^h)\to\mathcal{H}^h$ is a bounded operator descends from the finite dimension of $\mathcal{H}^h$, where every norm is equivalent. 
Indeed, if $\|\cdot\|$ is a norm on $\mathcal{H}^h$, there exists  constants $c$ and $C$ so that $\|H(\alpha)\|\leq c \sum_{i=1}^r|(\alpha,\omega_i)|{\|\omega_i\|_{L^2}}\le c \sum_{i=1}^r|(\alpha,\omega_i)|\le C\|\alpha\|_{L^p}$. 
\end{proof}

Hence, at first we infer that
$$L^p(M,E_0^h)=\mathcal{H}^h\oplus\left(\left(\mathcal{H}^h\right)^\bot\cap L^p(M,E_0^h)\right).$$

Throughout this section we  omit to specify that again $M$ is assumed to be compact and $1\le h\le 2n$.
Imitating the line of the proof presented in Appendix A, we start by obtaining an equivalent norm in the spaces $W^{\ell,p}(M,E_0^h)\cap\left(\mathcal{H}^h\right)^\bot$.

\begin{teo}\label{teo.scottstimaperpoinc}
 Let $1<p<\infty$. Let $\ell=2$ if $h=n,n+1$ and $\ell=1$ for other index $h$. Then, there exists $C_{M}>0$ such that, for every $\alpha\in W^{\ell,p}(M,E_0^h)\cap\left(\mathcal{H}^h\right)^\bot$, we have
\begin{gather}\label{eq.poincdebole}
\begin{split}
\|\alpha\|_{W^{1,p}(M,E_0^h)}&\leq C_{M}\Big(\|\d_C\alpha\|_{L^{p}(M,E_0^{h+1})}+\|\delta_C\alpha\|_{L^{p}(M,E_0^{h-1})}\Big),\quad\text{if $h\neq n,n+1$}, \\
\|\alpha\|_{W^{2,p}(M,E_0^n)}&\leq C_{M}\Big(\|\d_C\alpha\|_{L^{p}(M,E_0^{n+1})}+\|\d_C\delta_C\alpha\|_{L^{p}(M,E_0^{n})}\Big),\quad\text{if $h=n$}, \\
\|\alpha\|_{W^{2,p}(M,E_0^{n+1})}&\leq C_{M}\Big(\|\delta_C\alpha\|_{L^{p}(M,E_0^n)}+\|\delta_C\d_C\alpha\|_{L^{p}(M,E_0^{n+1})}\Big),\quad\text{if $h=n+1$}.
\end{split}
\end{gather}
\end{teo}
\begin{proof}
Without loss of generality we can assume $\alpha\neq 0$ and
we argue by contradiction.

\underline{If $h\neq n,n+1$}, we suppose that \eqref{eq.poincdebole} is not true. Then, for every $k\in\N$, there exists $\alpha_k\in W^{1,p}(M,E_0^h)\cap\left(\mathcal{H}^h\right)^\bot$ such that 
\begin{equation}\label{eq.randomstimak}
\|\alpha_k\|_{W^{1,p}}> k\Big(\|\d_C\alpha_k\|_{L^{p}}+\|\delta_C\alpha_k\|_{L^{p}}\Big).
\end{equation}
We can suppose that $\|\alpha_k\|_{L^p}=1$, for every $k$. 
By Theorem \ref{teo.GaffIneqcontman}, we have that
\begin{align*}
\|\alpha_k\|_{W^{1,p}}\leq C\Big(\|\alpha_k\|_{L^p}+\|\d_C\alpha_k\|_{L^{p}}+\|\delta_C\alpha_k\|_{L^{p}}\Big)\stackrel{\eqref{eq.randomstimak}}{\leq} C\|\alpha_k\|_{L^p}+\frac{C}{k}\|\alpha_k\|_{W^{1,p}}.
\end{align*}
Thus,
$$\|\alpha_k\|_{W^{1,p}}\leq\frac{kC}{k-C}\|\alpha_k\|_{L^p}=\frac{kC}{k-C}.$$
If $k\gg 1$, since $\frac{kC}{k-C}\xrightarrow{k\to +\infty} C$, then $\|\alpha_k\|_{W^{1,p}}\leq C+1$. Using \eqref{eq.randomstimak}, we obtain
\begin{equation}\label{eq.stimadeltadug0}
\|\d_C\alpha_k\|_{L^{p}}+\|\delta_C\alpha_k\|_{L^{p}}<\frac{C+1}{k},\quad\quad\text{for $k\gg 1$}.
\end{equation}
This implies that $(\delta_C\alpha_k)_{k\in\N}$ and $(\d_C\alpha_k)_{k\in\N}$ are bounded sequences in $L^{p}$, and thus (by Theorem \ref{teo.GaffIneqcontman} and $\|\alpha_k\|_{L^p}=1$) $(\alpha_k)_{k\in\N}$ is bounded in $W^{1,p}$. Up to pass to a subsequence, $\alpha_k\rightharpoonup \alpha\in W^{1,p}$, and, since $\alpha_k\in\left(\mathcal{H}^h\right)^\bot$, $\alpha$ is orthogonal to harmonic forms. But, by \eqref{eq.stimadeltadug0}, we have that
$$\delta_C\alpha_k\to 0\quad\text{and}\quad\d_C\alpha_k\to 0\quad\quad\text{in $L^{p}$},$$
proving that $\delta_C\alpha=\d_C\alpha=0$ weakly. Thus $\alpha$ is harmonic. Hence $\alpha=0$, but this is an absurd since we have supposed that $\alpha\neq 0$.\medskip

\underline{If $h=n$ or $h=n+1$}, we can argue analogously using the estimates given in Theorem \ref{teo.GaffIneqcontman} for $W^{2,p}$.
\end{proof}


\begin{coro}\label{prop.dWclosedsett}
Let $1<p<\infty$. Let $s=2$ if $h=n+1$ and $s=1$ otherwise, and $r=2$ if $h=n$ and $r=1$ otherwise. The spaces $\d_C\Big(W^{s,p}(M,E_0^{h-1})\Big), \delta_C\Big(W^{r,p}(M,E_0^{h+1})\Big)$ are closed under the $L^p$-convergence.
\end{coro}
\begin{proof}
Let $\alpha\in \overline{\d_C\left(W^{s,p}(M,E_0^{h-1})\right)}$. Thus, there exists a sequence $(\d_C\omega_k)_{k\in\N}\subseteq\d_C\Big(W^{s,p}(M,E_0^{h-1})\Big)$ such that $\d_C \omega_k\to \alpha$ in $L^p(M,E_0^h)$. Since $C^\infty(M,E_0^{h-1})$ is dense in $W^{s,p}(M,E_0^{h-1})$, for every $k\in\N$, there exists a sequence $(\eta_j^{(k)})_{j\in\N}\subseteq C^\infty(M,E_0^{h-1})$ such that $\eta_j^{(k)}\xrightarrow{j\to+\infty}\omega_k$ in $W^{s,p}(M,E_0^{h-1})$. Then, we define $\eta_k:=\eta_{j_k}^{(k)}$ such that
$\|\eta_{j_k}^{(k)}-\omega_k\|_{W^{s,p}}\leq \frac{1}{k}$, for every $k$, implying that
$$\|\d_C\eta_k-\alpha\|_{L^p}\leq\|\d_C(\eta_k-\omega_k)\|_{L^p}+\|\d_C\omega_k-\alpha\|_{L^p}\leq \frac{1}{k}+\|\d_C\omega_k-\alpha\|_{L^p}\xrightarrow{k\to+\infty} 0.$$
Invoking the decomposition for smooth forms, we can write $\eta_k=H(\eta_k)+\d_C\eta'_k+\delta_C\eta''_k$ and we can suppose $H(\eta_k)=0$ since $\d_C\eta_k=\d_C(\eta_k-H(\eta_k))=\d_C\delta_C\eta''_k$.

If $h\neq n,n+1$ (the cases $h=n$ or $h=n+1$ are analogous), by Theorem \ref{teo.scottstimaperpoinc}, we have
\begin{align*}
\|\delta_C\eta''_k\|_{W^{1,p}}&\leq C_M \left(\|\d_C\delta_C\eta''_k\|_{L^p}+\|\delta_C^2\eta''_k\|_{L^p}\right)\\
&=C_M\|\d_C\eta_k\|_{L^p}\leq C\left(\|\alpha\|_{L^p}+1\right),\quad\text{for $k\gg 1$}.
\end{align*}
By Banach-Alaoglu theorem, there exists $\gamma\in W^{1,p}(M,E_0^{h-1})$ such that, up to pass to a subsequence, $\delta_C\eta''_k\rightharpoonup \gamma$ in $W^{1,p}(M,E_0^{h-1})$. Thus, $\d_C\eta_k=\d_C\delta_C\eta''_k\rightharpoonup \d_C\gamma$ in $L^p(M,E_0^h)$. By the lower semicontinuity of $\|\cdot\|_{L^p}$,
$$\|\d_C\gamma-\alpha\|_{L^p}\leq \liminf_{k\to+\infty}\|\d_C\eta_k-\alpha\|_{L^p}=0.$$
We have proved that $\alpha=\d_C\gamma$, with $\gamma\in W^{1,p}(M,E_0^{h-1})$, i.e.\,$\d_C\Big(W^{s,p}(M,E_0^{h-1})\Big)$ is closed under the $L^p$-convergence.
We can argue analogously for $\delta_C\Big(W^{r,p}(M,E_0^{h+1})\Big)$, completing the proof.
\end{proof}

\medskip

Our approach now divides  in two parts. First we examinate the case $2<p<\infty$, where we can use the information of previous section: if $\alpha\in L^p(M,E_0^h)\cap\left(\mathcal{H}^h\right)^\bot$, since $p>2$ and $M$ is compact, $\alpha\in L^2(M,E_0^h)\cap\left(\mathcal{H}^h\right)^\bot$.

On the other hand, the case $1<p<2$ will follow after a duality argument on smooth functions in Proposition \ref{prop.regLp1mpm2} (see Proposition 5.17 in \cite{scott}).

\bigskip

{\bf Case $2<p<\infty$.}

\medskip

Using Sobolev embeddings in Proposition \ref{teo.Sobembthcontman}, and Remark \ref{oss.SobembstimaCQ}, we have the following result on the $W^{2\ell,p}$-regularity of the potential, for $2<p<\infty$.

\begin{teo}\label{teo.Lpreg2mpminfty}
If $2<p<\infty$ then for every $\alpha\in L^p(M,E_0^h)\cap\left(\mathcal{H}^h\right)^\bot$ there exists a unique $\gamma\in W^{2\ell,p}(M,E_0^h)\cap\left(\mathcal{H}^h\right)^\bot$ such that $\Delta_M\gamma=\alpha$.
Moreover, we have that 
$$\|\gamma\|_{W^{2\ell,p}}\leq c \|\alpha\|_{L^p},$$
for some constant $c>0$.
\end{teo}
\begin{proof}
 Since $p>2$ and $M$ is compact, $\alpha\in L^2(M,E_0^h)\cap\left(\mathcal{H}^h\right)^\bot$, that implies, by Theorem \ref{teo.greenopregsol}, the existence of a unique $\gamma\in W^{2\ell,2}(M,E_0^h)\cap\left(\mathcal{H}^h\right)^\bot$ such that $\Delta_M\gamma=\alpha$.

We prove now the   regularity of $\gamma$ in $W^{2\ell,p}$. Using Remark \ref{oss.SobembstimaCQ} with $p_1:=2$ and $q_1:=\min\left(p_1+\frac{p_1}{C_Q}, p\right)$, and by the continuity of Green's operator for $L^2$-forms, we obtain that 
$$\|\gamma\|_{L^{q_1}}\leq c\|\gamma\|_{W^{\ell,p_1}}\leq c\|\alpha\|_{L^{p_1}}\leq c\|\alpha\|_{L^{q_1}}\leq c\|\alpha\|_{L^p}<\infty.$$
Thus, $\gamma\in L^{q_1}$ and $\|\gamma\|_{L^{q_1}}\leq c\|\alpha\|_{L^{q_1}}$. By  Theorem \ref{teo.globalmaxhypocont}), we have
$$\|\gamma\|_{W^{2\ell,q_1}}\leq c\left(\|\Delta_M\gamma\|_{L^{q_1}}+\|\gamma\|_{L^{q_1}}\right)=c\left(\|\alpha\|_{L^{q_1}}+\|\gamma\|_{L^{q_1}}\right)\leq c\|\alpha\|_{L^p},$$
showing that $\gamma\in W^{2\ell,q_1}$.

If $q_1=p$, then the theorem is proved. Otherwise, we can argue analogously with $p_2:=q_1$ and $q_2:=\min\left(p_2+\frac{p_2}{C_Q},p\right)$, and, after a finite number of steps $k$ we have $q_k>2+\frac{2k}{C_Q}$, hence there exists $k$ so that $q_k>p$ and we obtain that $\gamma\in W^{2\ell,p}$ and $\|\gamma\|_{W^{2\ell,p}}\leq c\|\alpha\|_{L^p}$, completing the proof.
\end{proof}

If $2<p<\infty$, the operator
$$G:L^p(M,E_0^\bullet)\longrightarrow W^{2\ell,p}(M,E_0^\bullet)\cap\left(\mathcal{H}^h\right)^\bot,$$
such that $G(\alpha)$ is the unique solution of $\Delta_{M}G(\alpha)=\alpha-H(\alpha)$, is well defined and it is a linear continuous operator (thanks  to Theorem \ref{teo.Lpreg2mpminfty}). We call $G$ the Green operator associated to $\Delta_{M}$.
\begin{oss}\label{oss.GopsimmCinf}
The  $G$ is a symmetric operator on $C^\infty$. Indeed, if $\alpha,\beta\in C^\infty(M,E_0^h)$, recalling the Hodge decomposition obtained for smooth forms and that $G(\alpha),G(\beta) \in\left(\mathcal{H}^h\right)^\bot$, then
\begin{align*}
(G(\alpha),\beta)&=(G(\alpha),H(\beta)+\Delta_MG(\beta))=(G(\alpha),\Delta_MG(\beta))\\
&=(\Delta_MG(\alpha),G(\beta))=(H(\alpha)+\Delta_MG(\alpha),G(\beta))=(\alpha,G(\beta)).
\end{align*}
\end{oss}

\bigskip

{\bf Case $1<p<2$.}

\medskip

We now study the $L^p$-regularity of the potential $\gamma$ when $1<p<2$.
Keeping in mind Definition \ref{defi.GreenopCinf},
the operator $G:C^\infty(M,E_0^h)\longrightarrow C^\infty(M,E_0^h)\cap\left(\mathcal{H}^h\right)^\bot\ .$
We prove the following estimate (see also Proposition 5.17 in \cite{scott}).
\begin{prop}\label{prop.regLp1mpm2}
Let $1<p<2$ and  $\ell=1 $ if $h\neq n, n+1$ and $\ell=2$ for $h=n,n+1$. There exists a constant $c>0$, such that, for every $\alpha\in C^\infty(M,E_0^h)$, we have
$$\|G(\alpha)\|_{W^{2\ell,p}(M,E_0^h)}\leq c\|\alpha\|_{L^p(M,E_0^h)}.$$
\end{prop}
\begin{proof}
Let $\alpha\in C^\infty(M,E_0^h)$ and let $G(\alpha)\in C^\infty(M,E_0^h)\cap\left(\mathcal{H}^h\right)^\bot$. Since $M$ is compact, $\alpha\in L^p(M,E_0^h)$ and $G(\alpha)\in W^{2\ell,p}(M,E_0^h)\cap\left(\mathcal{H}^h\right)^\bot$.
Let 
$$\beta_k:=G(\alpha)\left(|G(\alpha)|^2+\frac{1}{k}\right)^{\frac{p-2}{2}}\in C^\infty(M,E_0^h),\quad\forall\,k\in\N.$$
Then, if $\frac{1}{p}+\frac{1}{p'}=1$, 
$$\|\beta_k\|_{L^{p'}}^{p'}=\int_M\,|\beta_k|^{p'}dV=\int_M\,|G(\alpha)|^{p'}\,\left(|G(\alpha)|^2+\frac{1}{k}\right)^{\frac{p-2}{2}p'}\,\ dV.$$
Recalling that $1<p<2$, we have that $\frac{p-2}{2}p'<0$, so that $|\beta_k|^{p'}\leq |G(\alpha)|^p\in L^1(M,E_0^h)$ for every $k\in\N$, and $|\beta_k|^{p'}\to|G(\alpha)|^p$ as $k\to+\infty$. By the dominated convergece theorem,
$$\|\beta_k\|_{L^{p'}}^{p'}\xrightarrow{k\to+\infty}\|G(\alpha)\|_{L^p}^p,$$
and that
$$|(G(\alpha),\beta_k)|=\left|\int_M\,<G(\alpha),G(\alpha)>\,\left(|G(\alpha)|^2+\frac{1}{k}\right)^{\frac{p-2}{2}}\,\ dV\right| \xrightarrow{k\to+\infty} \|G(\alpha)\|_{L^p}^p.$$
Thus, for every $\e>0$, there exists $k_\e\in\N$ such that $\|G(\alpha)\|_{L^p}^p\leq |(G(\alpha),\beta_k)|+\e$, for every $k\geq k_\e$.
Using the continuity of  operator $G$ for $L^{p'}$-forms (since $p'>2$), we have
\begin{align*}
\|G(\alpha)\|_{L^p}^p&\leq |(G(\alpha),\beta_k)|+\e \stackrel{\text{Remark \ref{oss.GopsimmCinf}}}{=} |(\alpha,G(\beta_k))|+\e \leq c\|\alpha\|_{L^p}\|G(\beta_k)\|_{L^{p'}}+\e\\
&\leq c \|\alpha\|_{L^p}\|\beta_k\|_{L^{p'}}+\e\xrightarrow{k\to+\infty} c \|\alpha\|_{L^p}\|G(\alpha)\|_{L^{p}}^{p/p'}+\e\xrightarrow{\e\to 0} c \|\alpha\|_{L^p}\|G(\alpha)\|_{L^{p}}^{p/p'}.
\end{align*}
Thus, $\|G(\alpha)\|_{L^p}\leq c \|\alpha\|_{L^p}$. By Theorem \ref{teo.globalmaxhypocont} and by the continuity of harmonic projection in $L^p$, we have
\begin{align*}
\|G(\alpha)\|_{W^{2\ell,p}}&\leq c\left(\|\Delta_M G(\alpha)\|_{L^p}+\|G(\alpha)\|_{L^p}\right)= c\left(\|\alpha-H(\alpha)\|_{L^p}+\|G(\alpha)\|_{L^p}\right)\\
&\leq c\left(\|\alpha\|_{L^p}+\|H(\alpha)\|_{L^p}\right)\leq c\|\alpha\|_{L^p},
\end{align*}
completing the proof.
\end{proof}
Since $C^\infty$ is dense in $L^p$, we may extend
$$G:L^p(M,E_0^h)\longrightarrow W^{2\ell,p}(M,E_0^h)\cap\left(\mathcal{H}^h\right)^\bot$$
as the unique continuous extension (by Proposition \ref{prop.regLp1mpm2}) of the operator  $G:C^\infty(M,E_0^h)\longrightarrow W^{2\ell,p}(M,E_0^h)\cap\left(\mathcal{H}^h\right)^\bot.$
\begin{oss}\label{venerdi}
If $1<p<2$ and $\alpha\in L^p(M,E_0^h)$, then $\Delta_M G(\alpha)=\alpha-H(\alpha)$, where $G$ is the operator defined above.
\end{oss}
\begin{proof}
 Indeed, let $(\alpha_k)_k\subseteq C^\infty(M,E_0^h)$ such that $\alpha_k\to\alpha$ in $L^p$. Then, $\Delta_M G(\alpha_k)=\alpha_k-H(\alpha_k)$ since $\alpha_k$ is a smooth form, for every $k\in\N$, and, for every $\varphi\in C^\infty(M,E_0^h)$,
\begin{align*}
0&=(\Delta_{M}G(\alpha_k)-(\alpha_k-H(\alpha_k)), \varphi)\\
&=(G(\alpha_k),\Delta_M\varphi)-(\alpha_k-H(\alpha_k),\varphi)\xrightarrow{k\to+\infty}(G(\alpha),\Delta_M\varphi)-(\alpha-H(\alpha),\varphi).
\end{align*}
Thus, $\Delta_M G(\alpha)=\alpha-H(\alpha)$.
\end{proof}

\medskip

Putting together every remarks on $L^p$-forms for the case $1<p<2$ and the case $2<p<\infty$, we have obtained a regularity result for the potential of a form  in $L^p(M,E_0^h)\cap\left(\mathcal{H}^h\right)^\bot$ as indicated in the next statement.
\begin{teo}\label{teo.Lpreggeneral}
Let  $1<p<\infty$. For every $\alpha\in L^p(M,E_0^h)\cap\left(\mathcal{H}^h\right)^\bot$ there exists a unique $\gamma\in W^{2\ell,p}(M,E_0^h)\cap\left(\mathcal{H}^h\right)^\bot$ such that $\Delta_M\gamma=\alpha$.
Moreover, 
$$\|\gamma\|_{W^{2\ell,p}}\leq c \|\alpha\|_{L^p}.$$
Equivalently, the operator$G$ is well defined on $L^p$-forms as
$$G:L^p(M,E_0^h)\longrightarrow W^{2\ell,p}(M,E_0^h)\cap\left(\mathcal{H}^h\right)^\bot$$
such that $\Delta_M G(\alpha)=\alpha-H(\alpha)$ and $G$ is a bounded linear operator.
\end{teo}\bigskip
By an argument analogous to the one used in proof of Lemma \ref{lemma.orthLpforms} we can obtain also the orthogonality on $L^p$-forms.
\begin{lemma}\label{lemma.orthLpforms bis}
Let $1<p<\infty$. We consider $s=2$ if $h=n+1$, and $s=1$ otherwise. In addition let $r=2$ if $h=n$ and $r=1$ otherwise. If $\alpha\in W^{s,p}(M,E_0^{h-1})$, $\beta\in W^{r,p}(M,E_0^{h+1})$ and $\xi\in\mathcal{H}^h$ such that $0=\d_C\alpha+\delta_C\beta+\xi$, then $\d_C\alpha=\delta_C\beta=\xi=0$.
\end{lemma}

\medskip

\begin{proof}[Proof of Theorem \ref{teo.HodgedecthLp}]
Let $\alpha\in L^p(M,E_0^h)$. By Theorem \ref{teo.Lpreggeneral},  $G(\alpha)\in W^{2\ell,p}(M,E_0^h)\cap\left(\mathcal{H}^h\right)^\bot$. An argument analogous to the one used for the smooth forms gives the  the decompositions \eqref{eq.decompLphodge} (again, $\ell=1$ in degree $h\neq n, n+1$ and $\ell=2$ in middle degrees).
We prove the uniqueness. We already know that $H(\alpha)$ is unique, thus if $\alpha$ can be decomposed as
$$\alpha=\d_C\omega+\delta_C\zeta+H(\alpha)=\d_C\widetilde\omega+\delta_C\widetilde\zeta+H(\alpha),$$
then $\d_C(\omega-\widetilde\omega)+\delta_C(\zeta-\widetilde\zeta)=0$. By Lemma \ref{lemma.orthLpforms bis} get the uniqueness.
\end{proof}
As for smooth forms, the previous decomposition can be written in the following way:
\begin{equation}\label{eq.decompLppWGhodge}
L^p(M,E_0^h)=\mathcal{H}^h\oplus\left\{
\begin{array}{ll}
\d_C\delta_C G\left(L^p(M,E_0^h)\right)\oplus\delta_C\d_C G\left(L^p(M,E_0^h)\right),&\text{if $h\neq n,n+1$},\\
\d_C\delta_C\d_C\delta_C G\left(L^p(M,E_0^h)\right)\oplus\delta_C\d_C G\left(L^p(M,E_0^h)\right),&\text{if $h=n$},\\
\d_C\delta_C G\left(L^p(M,E_0^h)\right)\oplus\delta_C\d_C\delta_C\d_C G\left(L^p(M,E_0^h)\right),&\text{if $h=n+1$}.
\end{array}
\right. 
\end{equation}
%
%
%

\subsection{Poincaré-type inequalities for $L^p$-Rumin forms}\label{sec:poincineq}

As a corollary of our main results, we can establish a more refined version of the inequalities that we proved in Theorem \ref{teo.scottstimaperpoinc}.
\begin{teo}\label{teo.poincineqcontman}
Let $(M,H,g^M)$ be a sub-Riemannian compact contact manifold without boundary,  and let $1<p<\infty$. Then, there exists $C_{p,M}>0$ such that
\begin{itemize}
	\item[i)]
for every $\alpha\in W^{1,p}(M,E_0^h)$,
$$
\|\alpha-H(\alpha)\|_{W^{1,p}(M,E_0^h)}\leq C_{p,M}\Big(\|\d_C\alpha\|_{L^p(M,E_0^{h+1})}+\|\delta_C\alpha\|_{L^p(M,E_0^{h-1})}\Big),\quad\text{if $h\neq n,n+1$}; 
$$
\item[ii)] for every $\alpha\in W^{2,p}(M,E_0^n)$,
$$
\|\alpha-H(\alpha)\|_{W^{2,p}(M,E_0^n)}\leq C_{p,M}\Big(\|\d_C\alpha\|_{L^p(M,E_0^{n+1})}+\|\d_C\delta_C\alpha\|_{L^p(M,E_0^n)}\Big); $$
\item[iii)] for every $\alpha\in W^{2,p}(M,E_0^{n+1})$,
$$
\|\alpha-H(\alpha)\|_{W^{2,p}(M,E_0^{n+1})}\leq C_{p,M}\Big(\|\delta_C\alpha\|_{L^p(M,E_0^n)}+\|\delta_C\d_C\alpha\|_{L^p(M,E_0^{n+1})}\Big).
$$
\end{itemize}
\end{teo}
\begin{proof}
Let $\alpha\in W^{\ell,p}(M,E_0^h)$, with $\ell=1$ $h\neq n,n+1$ and $2$ otherwise.
By Theorem \ref{teo.HodgedecthLp}, there exist $\omega\in W^{s,p}(M,E_0^{h-1})$ and $\zeta\in W^{r,p}(M,E_0^{h+1})$ such that
 \begin{equation}\label{prova}
	 \alpha=H(\alpha)+\d_C\omega+\delta_C\zeta,
 \end{equation}
with $s,r$ equals $1$ or $2$ as specified below, depending on the degree $h$.
\noindent{Case $h\neq n,n+1$}. We already know that $\omega , \zeta\in W^{1,p}\cap\left(\mathcal{H}\right)^\bot$. Since $\alpha\in W^{1,p}(M,E_0^h)$, we have that $\d_C(\delta_C\zeta)=\d_C\alpha\in L^p(M,E_0^{h+1})$ and $\delta_C(\d_C\omega)=\delta_C\alpha\in L^p(M,E_0^{h-1})$. Moreover, $\d_C(\d_C\omega)=\delta_C(\delta_C\zeta)=0$, implying that $\d_C\omega, \delta_C\zeta\in \mathcal{L}^{1,p}(M,E_0^h)$. By Theorem \ref{spazi uguali}, we have $\d_C\omega, \delta_C\zeta\in W^{1,p}\cap\left(\mathcal{H}\right)^\bot$.

Hence, by Theorem \ref{teo.scottstimaperpoinc} to $\d_C\omega$ and $\delta_C\zeta$:
\begin{align*}
\|\alpha-H(\alpha)\|_{W^{1,p}}&\leq\|\d_C\omega\|_{W^{1,p}}+\|\delta_C\zeta\|_{W^{1,p}} \\
&{\leq} C_{p,M}\Big(\|\d^2_C\omega\|_{L^p}+\|\delta_C\d_C\omega\|_{L^p}+\|\d_C\delta_C\zeta\|_{L^p}+\|\delta^2_C\d_C\zeta\|_{L^p_M}\Big) \\
&=C_{p,M}\Big(\|\delta_C\d_C\omega\|_{L^p_M}+\|\d_C\delta_C\zeta\|_{L^p_M}\Big) \\
&= C_{p,M}\Big(\|\delta_C\alpha\|_{L^p_M}+\|\d_C\alpha\|_{L^p_M}\Big).
\end{align*}

\noindent If $h=n$ (the case $h=n+1$ is analogous), arguing as in previous case, we have that $\d_C\omega,\delta_C\gamma\in W^{2,p}(M,E_0^n)$. Thus,
\begin{align*}
\|\alpha-H(\alpha)\|_{W^{2,p}}&\leq\|\d_C\omega\|_{W^{2,p}}+\|\delta_C\zeta\|_{W^{2,p}} \\
&{\leq} C_{p,M}\Big(\|\d^2_C\omega\|_{L^p}+\|\d_C\delta_C\d_C\omega\|_{L^p}+\|\d_C\delta_C\zeta\|_{L^p}+\|\d_C\delta^2_C\zeta\|_{L^p}\Big) \\
&=C_{p,M}\Big(\|\d_C\delta_C\d_C\omega\|_{L^p}+\|\d_C\delta_C\zeta\|_{L^p}\Big) \\
&= C_{p,M}\Big(\|\d_C\delta_C\alpha\|_{L^p}+\|\d_C\alpha\|_{L^p}\Big),
\end{align*}
completing the proof.
\end{proof}

\section{Appendix A: Morrey-type variational approach for $L^2$-decomposition}\label{appendix}
With the notation introduced in Section \ref{HodgeL2}, we add here the proofs of Proposition \ref{teo.findimkerlap} and Theorem \ref{teo.733morreyexistence}. Both proofs are very related to the Gaffney inequality given in Theorem \ref{teo.GaffIneqcontman}. We still assume that  $(M,H,g^M)$ is a sub-Riemannian compact contact manifold without boundary.

 First, thanks to Theorem \ref{teo.GaffIneqcontman}, we can define another norm on $W^{\ell,2}(M,E_0^\bullet)$, equivalent to the one defined previously. More precisely, we consider the equivalent norms $||\cdot||_G$ on $W^{\ell,2}(M,E_0^h)$ given by
\begin{equation*}
||\alpha||_G:=\left\{
\begin{array}{ll}
\|\d_C\alpha\|_{L^2}+\|\delta_C\alpha\|_{L^2}+\|\alpha\|_{L^2},\quad&\text{if $h\neq n,n+1$}, \\
\|\d_C\alpha\|_{L^2}+\|\d_C\delta_C\alpha\|_{L^2}+\|\alpha\|_{L^2},\quad&\text{if $h=n$}, \\
\|\delta_C\d_C\alpha\|_{L^2}+\|\delta_C\alpha\|_{L^2}+\|\alpha\|_{L^2},\quad&\text{if $h=n+1$}.
\end{array}\right.
\end{equation*}
Due to the equivalence of norms, we remark that $||\cdot||_G$ is lower semicontinuous with respect to the weak convergence in $W^{\ell,2}(M,E_0^h)$.

We need the following minimizing lemma for the functional $||\cdot||_G$ (see Lemma $7.3.1$ in \cite{morrey}).
\begin{lemma}\label{lemma.minpropWl2}
Let $V\subseteq L^2(M,E_0^h)$ be a closed vector subspace. Then, either $V\cap W^{\ell,2}(M,E_0^h)=\emptyset$ or there exists $\alpha\in V\cap W^{\ell,2}(M,E_0^h)$ with $\|\alpha\|_{L^2}=1$ such that $\alpha$ minimizes $||\cdot||_G$ among $V\cap W^{\ell,2}(M,E_0^h)$.
\end{lemma}
\begin{proof}
Suppose $W:=V\cap W^{\ell,2}(M,E_0^h)\neq\emptyset$.
Let $(\alpha_k)_{k\in\N}\subseteq W$ be a minimizing sequence, i.e.\,$\dsy ||\alpha_k||_G\to\inf_W ||\cdot||_G$, such that $\|\alpha_k\|_{L^2}=1$. Then, by  Theorem \ref{teo.GaffIneqcontman}, 
$$\|\alpha_k\|_{W^{\ell,2}(M,E_0^h)}\leq C ||\alpha_k||_G\leq \widetilde C.$$
Thus, $(\alpha_k)_{k\in\N}$ is a bounded sequence in $W^{\ell,2}(M,E_0^h)$, that implies, up to pass to a subsequence, that $\alpha_k\rightharpoonup\alpha$ in $W^{\ell,2}(M,E_0^h)$. Since $||\cdot||_G$ is lower semicontinuous with respect to the weak convergence in $W^{\ell,2}(M,E_0^h)$, we have
$$\inf_W\,||\cdot||_G\leq ||\alpha||_G\leq \liminf_{k\to+\infty}\,||\alpha_k||_G=\inf_W\,||\cdot||_G,$$
completing the proof.
\end{proof}

The following results extends to our setting a result stated in \cite{morrey} (see Theorem $7.3.1$ therein).
\begin{prop}\label{teo.findimkerlap-appendice}
 Then, for every $h$, $\mathcal{H}^h$ has finite dimension.
\end{prop}
\begin{proof}
Let $\alpha_0\in W^{\ell,2}(M,E_0^h)$ be a minimizing element of $||\cdot||_G$ among $W_0:=W^{\ell,2}(M,E_0^h)$, with $\|\alpha_0\|_{L^2}=1$ (that exists by Lemma \ref{lemma.minpropWl2}, where $V:=W^{\ell,2}(M,E_0^h)$). Iterating this process, we can define $V_k:=\left(\Span\{\alpha_0,\ldots,\alpha_{k-1}\}\right)^\bot$ (where the orthogonality is in $L^2$), and there exists $\alpha_k\in W_k:=V_k\cap W^{\ell,2}(M,E_0^h)$ such that minimizes $||\cdot||_G$ among $W_k$, with $\|\alpha_k\|_{L^2}=1$. In this way, we can construct a sequence of forms in $W^{\ell,2}(M,E_0^h)$ such that they are orthogonal to each other with unit norm in $L^2$.

We observe that $0\leq||\alpha_0||_G\leq||\alpha_1||_G\leq\cdots$. Now, if $||\alpha_0||_G>\|\alpha_0\|_{L^2}$, then, by the definition of $||\cdot||_G$ and by Proposition \ref{prop.characharmRum}, we have that at least one between (for example if $h\neq n, n+1$) $\|\delta_C\alpha_0\|_{L^2}, \|\d_C\alpha_0\|_{L^2}$ is different from $0$, i.e.\,there are no harmonic forms in $W^{\ell,2}(M,E_0^h)$ that are not identically zero.

On the other hand, if we suppose that $\alpha_k\in\mathcal{H}^h$, for every $k\in\N$, then $||\alpha_k||_G=\|\alpha_k\|_{L^2}=1$, for every $k$. By Theorem \ref{teo.GaffIneqcontman}, $(\alpha_k)_{k\in\N}$ is a bounded sequence in $W^{\ell,2}(M,E_0^h)$, thus there exists $\alpha\in W^{\ell,2}(M,E_0^h)$ such that $\alpha_k\rightharpoonup\alpha$ in $W^{\ell,2}(M,E_0^h)$ and $\alpha_k\to\alpha$ in $L^2(M,E_0^h)$.
But, by construction of $V_k$, we have that $\alpha_k\bot\alpha_m$ in $L^2$, for every $k,m$. Then
$$\|\alpha_k-\alpha_m\|^2_{L^2}=\|\alpha_k\|^2_{L^2}+\|\alpha_m\|^2_{L^2}=2,$$
proving that $(\alpha_k)_{k\in\N}$ cannot converge to $\alpha$, since it is not a Cauchy sequence. This absurd says us that $\mathcal{H}^h$ has finite dimension.
\end{proof}

\begin{prop}\label{teo.scottstimaperpoinc uno}
  There exists $C_{M}>0$ such that, for every $\alpha\in W^{\ell,2}(M,E_0^h)\cap\left(\mathcal{H}^h\right)^\bot$, we have
\begin{gather}\label{eq.poincdebole uno}
\begin{split}
\|\alpha\|_{W^{1,2}(M,E_0^h)}&\leq C_{M}\Big(\|\d_C\alpha\|_{L^2(M,E_0^{h+1})}+\|\delta_C\alpha\|_{L^2(M,E_0^{h-1})}\Big),\quad\text{if $h\neq n,n+1$}, \\
\|\alpha\|_{W^{2,2}(M,E_0^n)}&\leq C_{M}\Big(\|\d_C\alpha\|_{L^2(M,E_0^{n+1})}+\|\d_C\delta_C\alpha\|_{L^2(M,E_0^{n})}\Big),\quad\text{if $h=n$}, \\
\|\alpha\|_{W^{2,2}(M,E_0^{n+1})}&\leq C_{M}\Big(\|\delta_C\alpha\|_{L^2(M,E_0^n)}+\|\delta_C\d_C\alpha\|_{L^2(M,E_0^{n+1})}\Big),\quad\text{if $h=n+1$}.
\end{split}
\end{gather}
\end{prop}
This proposition is a particular instance of Proposition \ref{teo.scottstimaperpoinc}, proved in Section 4.
Having in mind Proposition \ref{teo.scottstimaperpoinc uno} 
%
we define now the following Dirichlet-type integrals:
\begin{equation*}
||D\alpha||:=\left\{
\begin{array}{ll}
\left(\|\d_C\alpha\|^2_{L^2}+\|\delta_C\alpha\|^2_{L^2}\right)^{\frac{1}{2}},\quad&\text{if $h\neq n,n+1$}, \\
\left(\|\d_C\alpha\|^2_{L^2}+\|\d_C\delta_C\alpha\|^2_{L^2}\right)^{\frac{1}{2}},\quad&\text{if $h=n$}, \\
\left(\|\delta_C\d_C\alpha\|^2_{L^2}+\|\delta_C\alpha\|^2_{L^2}\right)^{\frac{1}{2}},\quad&\text{if $h=n+1$}.
\end{array}\right.
\end{equation*}
We observe that the Dirichlet integrals are norms in $W^{\ell,2}\cap\left(\mathcal{H}^h\right)^\bot$ due to Proposition \ref{teo.scottstimaperpoinc}. Indeed, the only thing that we have to verify is the positiveness. If $\alpha\in W^{\ell,2}\cap\left(\mathcal{H}^h\right)^\bot$ is such that $||D\alpha||=0$, then $$\|\d_C\alpha\|_{L^2}=\|\delta_C\alpha\|_{L^2}=0,$$ thus $\d_C\alpha=\delta_C\alpha=0$. By characterization of harmonic forms, we have that $\alpha\in\mathcal{H}^h$. On the other hand, we have that $\alpha\in\left(\mathcal{H}^h\right)^\bot$ by construction. Thus, $\alpha=0$.


Now, by a  Morrey-type variational approach in $L^2$ (see Theorem $7.3.3$ in \cite{morrey}) we are able to prove the existence result related to $\Delta_M$ stated in Theorem \ref{teo.733morreyexistence}.



\begin{teo}\label{teo.733morreyexistence-app}
Let $\alpha\in L^2(M,E_0^h)$ such that $\alpha$ is orthogonal to $\mathcal{H}^h$. Thus,
\begin{enumerate}
\item if $h\neq n,n+1$, there exists a unique form $\gamma\in W^{1,2}(M,E_0^h)\cap\left(\mathcal{H}^h\right)^\bot$ such that
\begin{equation}\label{eq.soldeboleLaphodgegen uno}
(\d_C\gamma,\d_C\zeta)+(\delta_C\gamma,\delta_C\zeta)=(\alpha,\zeta),\quad\quad\forall\,\zeta\in W^{1,2}(M,E_0^h);
\end{equation}
\item if $h=n$, there exists a unique form $\gamma\in W^{2,2}(M,E_0^n)\cap\left(\mathcal{H}^n\right)^\bot$ such that
\begin{equation}\label{eq.soldeboleLaphodgen due}
(\d_C\delta_C\gamma,\d_C\delta_C\zeta)+(\d_C\gamma,\d_C\zeta)=(\alpha,\zeta),\quad\quad\forall\,\zeta\in W^{2,2}(M,E_0^n);
\end{equation}
\item if $h=n+1$, there exists a unique form $\gamma\in W^{2,2}(M,E_0^{n+1})\cap\left(\mathcal{H}^{n+1}\right)^\bot$ such that
\begin{equation}\label{eq.soldeboleLaphodgenp1 tre}
(\delta_C\d_C\gamma,\delta_C\d_C\zeta)+(\delta_C\gamma,\delta_C\zeta)=(\alpha,\zeta),\quad\quad\forall\,\zeta\in W^{2,2}(M,E_0^{n+1}).
\end{equation}
\end{enumerate}
\end{teo}
\begin{proof}
Let $\ell=2$ if $h=n,n+1$ or $\ell=1$ otherwise. 

\textit{{Uniqueness}}. If $\gamma_1,\gamma_2\in W^{\ell,2}(M,E_0^h)$ such that satisfy \eqref{eq.soldeboleLaphodgegen uno}-\eqref{eq.soldeboleLaphodgenp1 tre}, then in particular
$$(\gamma_1-\gamma_2,\Delta_{M,h}\zeta)=0,\quad\quad \forall\,\zeta\in C^\infty(M,E_0^h),$$ 
i.e.\,$\Delta_{M,h}(\gamma_1-\gamma_2)=0$ in distributional sense. Hence $\gamma_1-\gamma_2$ is harmonic. On the other hand, by hypothesis, $\gamma_1-\gamma_2\in\left(\mathcal{H}^h\right)^\bot$, thus $\gamma_1=\gamma_2$.

%

\medskip
 \textit{{Existence}}.

{(Case $h\neq n, n+1$)}. We want to prove that the solution of \eqref{eq.soldeboleLaphodgegen uno} is a solution of the following minimization problem: 
\begin{equation}\label{eq.minprobteoH}
\min\Big\{I(\beta):=||D\beta||^2-2(\beta,\alpha)\,\,\Big|\,\, \beta\in W^{1,2}(M,E_0^h)\cap\left(\mathcal{H}^h\right)^\bot\Big\}.
\end{equation}
At first, we observe that $I$ is a lower semicontinuous functional with respect to the weak convergence in $W^{1,2}(M,E_0^h)$. Indeed,  since beeing orthogonal to the harmonic forms is stable under $W^{\ell,2}$-convergence, and since $||D\beta||$ is an equivalent norm in $W^{\ell,2}\cap\left(\mathcal{H}\right)^\bot$, we have that $||D\beta||$ is lower semicontinuous with respect to the weak convergence in $W^{\ell,2}$.
Hence, if $\beta, (\beta_k)_k\subseteq \beta\in W^{1,2}(M,E_0^h)\cap\left(\mathcal{H}^h\right)^\bot$ such that $\beta_k\rightharpoonup\beta$ in $W^{1,2}(M,E_0^h)$, since also $(\beta_k,\alpha)\to (\beta,\alpha)$, thus, 
\begin{align*}
I(\beta)&=\|\d_C\beta\|_{L^2}^2+\|\delta_C\beta\|_{L^2}^2-2(\beta,\alpha)\\
&\leq\liminf_{k\to+\infty}\Big(\|\d_C\beta_k\|_{L^2}^2+\|\delta_C\beta_k\|_{L^2}^2-2(\beta_k,\alpha)\Big)=\liminf_{k\to+\infty} I(\beta_k).
\end{align*}

Moreover, the functional $I$ is coercive. Indeed, if $\beta\in W^{1,2}(M,E_0^h)\cap\left(\mathcal{H}^h\right)^\bot$, by the Cauchy-Schwartz and Young inequalities, considering $1/C>0$ as the constant in Proposition \ref{teo.scottstimaperpoinc uno}, we have
\begin{equation}\label{eq.Youngineqhodgennn}
(\beta,\alpha)\leq\frac{C}{2}\|\beta\|^2_{L^2}+\frac{2}{C}\|\alpha\|^2_{L^2}.
\end{equation}
Thus, for every $\beta\in W^{1,2}(M,E_0^h)\cap\left(\mathcal{H}^h\right)^\bot$, we have proved
\begin{equation}\label{eq.stimaIbeta}
I(\beta)=||D\beta||^2-2(\beta,\alpha)\stackrel{\eqref{eq.Youngineqhodgennn}}{\geq}\frac{C}{2}||\beta||_{W^{1,2}}^2-\frac{2}{C}\|\alpha\|^2_{L^2}
\end{equation}
and
\begin{equation}\label{eq.stimaIinflimitato}
I(\beta)\geq-\frac{2}{C}\|\alpha\|^2_{L^2}\in\R,\quad\quad\forall\,\beta\in W^{1,2}(M,E_0^h)\cap\left(\mathcal{H}^h\right)^\bot,
\end{equation}
showing that $I$ is coercive. 
By direct methods in calculus of variations, we have that $I$ admits a minimum $\gamma$ among $\mathcal{A}$.
Now, for every $\zeta\in\mathcal{A}$ and for every $t\in\R$, we have
$$I(\gamma)\leq I(\gamma+t\zeta)=I(\gamma)+2t\Big((\d_C\gamma,\d_C\zeta)+(\delta_C\gamma,\delta_C\zeta)-(\alpha,\zeta)\Big)+t^2||D\zeta||^2,$$
and it is simple to prove that 
\begin{equation}\label{eq.risnonnnp1}(\delta_C\d_C\gamma,\delta_C\d_C\zeta)+(\delta_C\gamma,\delta_C\zeta)=(\alpha,\zeta).
\end{equation}
Recalling that $\alpha\in\left(\mathcal{H}^h\right)^\bot$, by Proposition \ref{prop.characharmRum}, we have that \eqref{eq.risnonnnp1} holds for every $\zeta\in W^{1,2}(M,E_0^h)$, completing the proof in the case $h\neq n, n+1$.

\noindent{(Case $h= n$)}.  We consider the analogous minimizing problem \eqref{eq.minprobteoH} with $\mathcal{A}:=W^{2,2}(M,E_0^h)\cap\left(\mathcal{H}^h\right)^\bot$. In the same way of the case $h\neq n, n+1$, we have the inequality
\begin{equation}\label{eq.stimaILSCinflimncase}
I(\beta)\geq\frac{C}{2}||\beta||_{W^{2,2}}^2-\frac{2}{C}\|\alpha\|^2_{L^2},\quad\quad\forall\,\beta\in W^{2,2}(M,E_0^h)\cap\left(\mathcal{H}^h\right)^\bot,
\end{equation}
and the functional $I$ is again a lower semicontinuous operator with respect to the weak convergence in $W^{2,2}(M,E_0^h)$ and coercive in $\mathcal{A}$. As previously argued, we obtain the the existence of a minimum $\gamma$ of $I$ in $\mathcal{A}$, and, for every $\zeta\in\mathcal{A}$ and for every $t\in\R$, we have
$$I(\gamma)\leq I(\gamma+t\zeta)=I(\gamma)+2t\Big((\d_C\gamma,\d_C\zeta)+(\d_C\delta_C\gamma,\d_C\delta_C\zeta)-(\alpha,\zeta)\Big)+t^2||D\zeta||^2.$$
Proceeding as before, we obtain the thesis for the case $h=n$. The case $h=n+1$ is similar to the previous one.
\end{proof}

\begin{oss}\label{oss.contweaksolL2-app}
The operator that for every $\alpha\in L^2(M,E_0^h)\cap\left(\mathcal{H}^h\right)^\bot$ associates $\gamma\in W^{\ell,2}(M,E_0^h)\cap\left(\mathcal{H}^h\right)^\bot$  is a continuous operator. Indeed, by \eqref{eq.stimaIbeta} and \eqref{eq.stimaILSCinflimncase}, since $\beta=0\in\mathcal{A}$ and $\gamma$ is the minimum of $I$ in $\mathcal{A}$, we have
$$\frac{C}{2}\|\gamma\|^2_{W^{\ell,2}(M,E_0^h)}-\frac{2}{C}\|\alpha\|^2_{L^2(M,E_0^h)}\leq I(\gamma)\leq I(0)=0.$$
Hence, for $\ell=2$ if $h=n,n+1$ and $\ell=1$ otherwise, we get
\begin{equation}\label{eq.contL2weksoll2-app}
\|\gamma\|_{W^{\ell,2}(M,E_0^h)}\leq\frac{2}{C}\|\alpha\|_{L^2(M,E_0^h)},\quad\quad\forall\,\alpha\in L^2(M,E_0^h)\cap\left(\mathcal{H}^h\right)^\bot.
\end{equation}
\end{oss}

\section{Appendix B: a consequence of the Gaffney-inequality}\label{appendix b}
We recall an inequalities proved in the setting of Heisenberg groups for differential forms in $C_0^\infty(\he{n}, E_0^h)$ (see \cite{BF7}, Remark 5.3, i), iii),vi) therein).
\begin{lemma}\label{gaffney-Hn-qui} Let $1\le h\le 2n$,   and $1<p<\infty$.
\begin{itemize}
\item[i)]  then there
exists $C>0$ such that for all $u \in {W^{1,p}(\he{n}, E_0^h)}$ with $h\neq n, n+1$
\begin{equation}\label{p>1:eq1}\begin{split}
\| u \|_{W^{1,p}(\he{n}, E_0^h)} &
\le C\big( \| \d_C u \|_{L^{p}(\he{n}, E_0^{h+1})} +
\| \delta_C u \|_{L^{p}(\he{n}, E_0^{h-1})}
\\& \hphantom{xxxxx}+ \|u\|_{L^p(\he{n}, E_0^h)}
\big).
\end{split}\end{equation}
Moreover,
\item[ii)]  if $h= n$,   and $1<p<\infty$, then there
exists $C>0$ such that for all $u \in {W^{2,p}(\he{n}, E_0^n)}$
\begin{equation}\label{treno}\begin{split}
\| u \|_{W^{2,p}(\he{n}, E_0^n)} &
\le C\big( \| \d_C u \|_{L^{p}(\he{n}, E_0^{n+1})} +
\| \d_C\delta_C u \|_{L^{p}(\he{n}, E_0^{n})}
\\&
 \hphantom{xxxxx} +  \|u\|_{L^p(\he{n}, E_0^n)}\big);
\end{split}\end{equation}
\item[iii)]  if $h=  n+1$,   and $1<p<\infty$, then there
exists $C>0$ such that for all $u \in {W^{2,p}_{\he{}}(\he{n}, E_0^{n+1})}$
\begin{equation}\label{p>1:eq2bis}\begin{split}
\| u \|_{W^{2,p}(\he{n}, E_0^{n+1})} &
\le C\big( \| \delta_C \d_C u \|_{L^{p}(\he{n}, E_0^{n+1})} +
\| \delta_C u \|_{L^{p}(\he{n}, E_0^{n})}
\\&
 \hphantom{xxxxx} +  \|u\|_{L^p(\he{n}, E_0^{n+1})}\big).
\end{split}\end{equation}
\end{itemize}
\end{lemma}
 group.
From the previous lemma we can  easily prove:
\begin{coro}\label{gaffney-Hn_Spazi_Sobolev-per equivalenza} Let $1<p<\infty$ and $1\le h\le 2n$. Let $u\in L^p(\he{n}, E_0^h)$ compactly supported.

\begin{itemize}
\item[i)] If $h\neq n, n+1$, suppose also  that $\d_C u\in L^{p}_{\he{}}(\he{n}, E_0^{h+1}),\, \, \delta_C u \in L^{p}(\he{n}, E_0^{h-1})$.  Then $u\in W^{1,p}(\he{n}, E_0^{h})$.
\item[ii)]
If $h=n$, suppose also   
 $\d_C u\in L^{p}(\he{n}, E_0^{n+1})$ and $\d_C\delta_C u \in L^{p}(\he{n}, E_0^{n})$. Then $u\in {W^{2,p}(\he{n}, E_0^n)} $.  
\item[iii)]
If $h=  n+1$, suppose also  $\delta_C\d_C u\in L^{p}(\he{n}, E_0^{n+1}),$ and $\delta_C u \in L^{p}(\he{n}, E_0^{n})$. Then
 $u\in {W^{2,p}(\he{n}, E_0^{n+1})}$
\end{itemize}
\end{coro}
\begin{proof} Let us prove i).
If $\eps<1$, let
$\omega_\epsilon$ be an usual
Friedrichs' mollifier (for the group structure) and  
let us consider  $\omega_\epsilon\ast u$. We have $\omega_\epsilon\ast u\to u$ in $L^p$,  $\d_C(\omega_\epsilon\ast u)=\omega_\epsilon\ast \d_C u\to \d_C u$ in in $L^p$ and $\delta_C(\omega_\epsilon\ast u)=\omega_\epsilon\ast \delta_C u\to \delta_C u$ in in $L^p$. Hence by i) in Lemma \ref{gaffney-Hn-qui} we have
\begin{equation}\label{regolarizzata}\begin{split}
\| \omega_\epsilon\ast u \|_{W^{1,p}(\he{n}, E_0^h)} &
\le C\big( \| \d_C(\omega_\epsilon\ast u) \|_{L^{p}(\he{n}, E_0^{h+1})} +
\| \delta_C (\omega_\epsilon\ast u) \|_{L^{p}(\he{n}, E_0^{h-1})}
\\& \hphantom{xxxxx}+ \|\omega_\epsilon\ast u\|_{L^p(\he{n}, E_0^h)}
\big).
\end{split}\end{equation}

To prove that $u\in {W^{1,p}(\he{n}, E_0^h)}$ we have to show that any components are in ${W^{1,p}(\he{n})}$.  

If $W$ is an horizontal derivative, then there exist a constant $c$ so that
$$
\|W(\omega_\epsilon\ast u )\|_{L^p}\le c
$$
uniformly on $\e$. We still denote a component of $\omega_\e\ast u$ with $\omega_\epsilon\ast u$ and a component of $u$ with $u$. Up to a subsequence, we have $W(\omega_\epsilon\ast u )\rightharpoonup v$ weakly-$L^p$ for some $L^p$-function $v$, and $\omega_\epsilon\ast u\to u$ in $L^p$. It is immediate to show  that $v=W u$ in a weak sense. Indeed, if $\varphi\in C_0^\infty(\he n)$,
$$
\int v\varphi=\lim_{\epsilon\to 0}\int W(\omega_\epsilon\ast u )\varphi=-\lim_{\epsilon\to 0}\int (\omega_\epsilon\ast u)\,W\varphi
$$
 By the lower semicontinuity of the norm
\begin{equation}\begin{split}
\|W u\|_{L^p}&\le \liminf_{\epsilon\to 0} \|W(\omega_\epsilon\ast u )\|_{L^p}
\le \liminf_{\epsilon\to 0} \|\omega_\epsilon\ast u\|_{W^{1,p}(\he{n}, E_0^h)}
\\&
\le C\liminf_{\epsilon\to 0}\Big( \| \d_C(\omega_\epsilon\ast u) \|_{L^p(\he{n}, E_0^{h+1})} +
\| \delta_C (\omega_\epsilon\ast u) \|_{L^p(\he{n}, E_0^{h-1})}
\\& \hphantom{xxxxxxxxxxx}+ \|\omega_\epsilon\ast u\|_{L^p(\he{n}, E_0^h)}
\Big)
\\& 
=C\big( \| \d_C u \|_{L^{p}(\he{n}, E_0^{h+1})} +
\| \delta_C  u \|_{L^{p}(\he{n}, E_0^{h-1})}
+ \| u\|_{L^p(\he{n}, E_0^h)}
\big)
\end{split}\end{equation}
Hence $u\in {W^{1,p}(\he{n}, E_0^h)}$ and we have proved i). The other case are similar.

\end{proof}
 We  prove the following result as a consequence of the previous statement in $\he{n}$. 
Let us set $$\mathcal{L}^{1,p}(M, E_0^h)=\{\alpha\in L^p(M, E_0^h): \d^M_C \alpha\in L^{p}(M, E_0^{h+1}),\, \, \delta^M_C \alpha \in L^{p}(M, E_0^{h-1})\}$$ when $h\neq n,n+1$ endowed with the norm
$$\|\alpha\|_{L^p(M,E_0^h)}+\|\d_C^M\alpha\|_{L^{p}(M,E_0^{h+1})}+\|\delta^M_C\alpha\|_{L^{p}(M,E_0^{h-1})}.$$

On the other hand, 
if $h=n$, we define the space 
$$\mathcal{L}_n^{2,p}(M, E_0^n)=\{\alpha\in L^p(M, E_0^n):\ \d^M_C \alpha\in L^{p}(M, E_0^{n+1}), \ \ \d^M_C\delta^M_C \alpha \in L^{p}(M, E_0^{n})\}$$ equipped with the norm  
$$\|\alpha\|_{L^p(M,E_0^n)}+\|\d_C^M\alpha\|_{L^{p}(M,E_0^{n+1})}+\|\d^M_C\delta^M_C\alpha\|_{L^{p}(M,E_0^{n})}.$$ 

Similarly, if 
 $h=n+1$, we set 
$$\mathcal{L}_{n+1}^{2,p}(M, E_0^{n+1})=\{\alpha\in L^p(M, E_0^{n+1}):\ \delta^M_C\d_C^M\alpha\in L^{p}(M, E_0^{n+1}), \ \ \delta^M_C \alpha \in L^{p}(M, E_0^{n})\}$$ equipped with the norm  
$$\|\alpha\|_{L^p(M,E_0^{n+1})}+\|\delta_C^M\alpha\|_{L^{p}(M,E_0^{n})}+\|\delta_C^M\d^M_C\alpha\|_{L^{p}(M,E_0^{n+1})}.$$ 

\begin{teo}\label{spazi uguali-appe}
 Let $1\leq h\leq 2n$, $1< p<\infty$. With the notation above, 
\begin{itemize}
	\item[i)] if $h\neq n, n+1$,  $\mathcal{L}^{1,p}(M, E_0^h)=W^{1,p}(M, E_0^h)$,
	\item[ii)] if $h= n$,  $\mathcal{L}_n^{2,p}(M, E_0^n)=W^{2,p}(M, E_0^n)$,
	\item[ii)] if $h= n+1$,  $\mathcal{L}_{n+1}^{2,p}(M, E_0^{n+1})=W^{2,p}(M, E_0^{n+1})$.
\end{itemize}
\end{teo}
\begin{proof} Reasoning as in \cite{BTT}, with the same choice of atlas used for proving Theorem 5.4 in \cite{BTT}, and Theorem 2.22 in this paper, we obtain the result.
Indeed let us consider for example in the case i). Let $\alpha$  be an arbitrary element of $\mathcal{L}^{1,p}(M, E_0^h)$. Using a partition of unity, we may assume that $\alpha$ is supported in $\psi_j(B(e,1))$. Let $\psi_j^\#\alpha$ denote its pullback, that is compactly supported. We can apply to $\psi_j^\#\alpha$ Corollary \ref{gaffney-Hn_Spazi_Sobolev-per equivalenza}. Thus, the techniques developped in \cite{BTT} apply also to the simpler case  $M$ compact.
\end{proof}

\section*{Acknowledgements}

A.B. is supported by the University of Bologna, funds for selected research topics, and 
 by PRIN 2022 F4F2LH - CUP J53D23003760006 ``Regularity problems in sub-Riemannian structures'', and by GNAMPA of INdAM (Istituto Nazionale di Alta Matematica ``F. Severi''), Italy.

The work of A.R.\,has been partially funded by the ICSC foundation - Centro Nazionale di Ricerca in High Performance Computing, Big Data and Quantum Computing.

\bibliographystyle{IEEEtranS}

\bibliography{BR_submitted}

\bigskip
\tiny{
\noindent
Annalisa Baldi
\par\noindent
Universit\`a di Bologna, Dipartimento
di Matematica\par\noindent Piazza di
Porta S.~Donato 5, 40126 Bologna, Italy.
\par\noindent
e-mail:
annalisa.baldi2@unibo.it 
}

\medskip

\tiny{
\noindent
Alessandro Rosa 
\par\noindent Istituto Nazionale di Fisica Nucleare (INFN)
\par\noindent  Via Bruno Rossi $1$, 50019 Sesto Fiorentino (FI), Italy.
\par\noindent 
e-mail: alessandro.rosa@fi.infn.it}


\end{document}